\documentclass[12pt]{amsart}

\usepackage[foot]{amsaddr}
\usepackage[english]{babel}
\usepackage[%pagebackref,
pdftex,hyperfootnotes]{hyperref}
\usepackage[usenames,dvipsnames,svgnames,table,x11names]{xcolor}
\usepackage{pgfplots}
\usepackage{tikz}
\usepackage{tikz-3dplot}
\usetikzlibrary{calc,shadows,shapes.callouts,shapes.geometric,shapes.misc,positioning,patterns,decorations.pathmorphing,decorations.markings,decorations.fractals,decorations.pathreplacing,shadings,fadings}
\usepackage[utf8]{inputenc}

\usepackage[total={16cm,20cm}, top=4.5cm, left=2.5cm]{geometry}
\usepackage{rotating,multirow,pdflscape}
\usepackage{amssymb,latexsym,amsmath,amsthm}
\usepackage{mathrsfs}
\usepackage{mathtools,arydshln,mathdots}
\usepackage{bigints}
\usepackage{comment}

\makeatletter
\renewcommand*\env@matrix[1][*\c@MaxMatrixCols c]{%
 \hskip -\arraycolsep
 \let\@ifnextchar\new@ifnextchar
 \array{#1}}
\makeatother

\mathtoolsset{centercolon}
\usepackage[table]{xcolor}
\usepackage[toc,page]{appendix}

\usepackage{tocvsec2}

\usepackage{Baskervaldx}
\usepackage{newtxmath}
\newtheorem{coro}{Corollary}

\newtheorem{defi}{Definition}
\newtheorem{teo}{Theorem}
\newtheorem{pro}{Proposition}
\newtheorem{lemma}{Lemma}

\newtheorem{rem}{Remark}

\newcommand{\ii}{\operatorname{i}}

\renewcommand{\d}{\operatorname{d}}

\newcommand{\diag}{\operatorname{diag}}

\newcommand{\C}{\mathbb{C}}

\newcommand{\N}{\mathbb{N}}
\newcommand{\R}{\mathbb{R}}

\makeatletter
\def\@settitle{\begin{center}%
  \baselineskip14\p@\relax
  \bfseries
  \uppercasenonmath\@title
  \@title
  \ifx\@subtitle\@empty\else
  \\[1ex]\uppercasenonmath\@subtitle
  \footnotesize\mdseries\@subtitle
  \fi
 \end{center}%
}
\def\subtitle#1{\gdef\@subtitle{#1}}
\def\@subtitle{}
\makeatother

\title[Matrix Biorthogonality: from Hermite to Painlevé]{Matrix Biorthogonal Polynomials: eigenvalue problems and non-Abelian discrete Painlevé equations.}

\subtitle{A Riemann--Hilbert Problem perspective}

\author[A Branquinho]{Amilcar Branquinho$^\dag$}
\address{$^\dag$Departamento de Matemática,
Universidade de Coimbra, 3001-454 Coimbra, Portugal}
\email{ajplb@mat.uc.pt}

\author[A Foulquié]{Ana Foulquié Moreno$^\maltese$}
\address{$^\maltese$Departamento de Matemática, Universidade de Aveiro, 3810-193 Aveiro, Portugal}
\email{foulquie@ua.pt}

\author[M Mañas]{Manuel Mañas$^\ddag$}
\address{$^\ddag$Departamento de Física Teórica, Universidad Complutense de Madrid, 28040-Madrid, Spain}
\email{manuel.manas@ucm.es}
\thanks{$^\dag$Acknowledges Centro de Matem\'{a}tica da Universidade de Coimbra (CMUC) -- UID/MAT/00324/2013, funded by the Portuguese Government through FCT/MEC and co-funded by the European Regional Development Fund through the Partnership Agreement PT2020}

\thanks{$^\maltese$Acknowledges CIDMA Center for Research and Development in Mathematics and Applications (University of Aveiro) and the 
Portuguese Foundation for Science and Technology (FCT)
within project UID/MAT/04106/2013}

\thanks{$^\ddag$Thanks financial support from the Spanish ``Ministerio de Economía y Competitividad" research project [MTM2015-65888-C4-2-P], \emph{Ortogonalidad, teoría de la aproximación y aplicaciones en física matemática}}

\begin{document}

\maketitle

\begin{abstract}
In this paper we use the Riemann--Hilbert problem, with jumps supported on appropriate curves in the complex plane, for matrix biorthogonal polynomials and apply it to find Sylvester systems of differential equations for the orthogonal polynomials and its second kind functions as well. For this aim, Sylvester type differential Pearson equations for the matrix of weights are shown to be instrumental. Several applications are given, in order of increasing complexity. First, a general discussion of non-Abelian Hermite biorthogonal polynomials in the real line, understood as those whose matrix of weights is a solution of a Sylvester type Pearson equation with coefficients first order matrix polynomials, is given. All these is applied to the discussion of possible scenarios leading to eigenvalue problems for second order linear differential operators with matrix eigenvalues. Nonlinear matrix difference equations are discussed next. Firstly, for the general Hermite situation a general non linear relation (non trivial because the non commutativity features of the setting) for the recursion coefficients is gotten.
 In the next case of higher difficulty, degree two polynomials are allowed in the Pearson equation, but the discussion is simplified by considering only a left Pearson equation. In the case, the support of the measure is on an appropriate branch of an hyperbola. The recursion coefficients are shown to fulfill a non-Abelian extension of the alternate discrete Painlev\'e I equation. Finally, a discussion is given for the case of degree three polynomials as coefficients in the left Pearson equation characterizing the matrix of weights. However, for simplicity only odd polynomials are allowed. In this case, a new and more general matrix extension of the discrete Painlev\'e I equation is found.
\end{abstract}

\tableofcontents

\section{Introduction}

Matrix extensions of real orthogonal polynomials where first discussed back in 1949 by Krein~\cite{Krein1,Krein2} 
and thereafter were studied sporadically until the last decade of the~XX~century, being some relevant papers~\cite{bere,geronimo,nikishin}. 
Then, in 1984, Aptekarev and Nikishin, for a kind of discrete Sturm--Liouville operators, solved the corresponding scattering problem in~\cite{nikishin}, and found that the polynomials that satisfy a relation of the~form
\begin{align*}
xP_{k}(x)&=A_{k}P_{k+1}(x)+B_{k}P_{k}(x)+A_{k-1}^{*}P_{k-1}(x),& k&=0,1,\dots,
\end{align*}
are orthogonal with respect to a positive definite measure; i.e., they derived a matrix version of Favard's theorem.

In a period of 20 years, from 1990 to 2010, it was found that matrix orthogonal polynomials (MOP) satisfy, in some cases, properties as do the classical orthogonal polynomials.
Let us mention, for example, that for matrix versions of Laguerre, Hermite and Jacobi polynomials, i.e., the scalar-type Rodrigues' formula~\cite{duran20052,duran_1} and a second order differential equation~\cite{borrego,duran_3,duran2004} has been discussed. It also has been proven~\cite{duran_5} that operators of the form
$D$=${\partial}^{2}F_{2}(t)$+${\partial}^{1}F_{1}(t)$+${\partial}^{0}F_{0}$ have as eigenfunctions different infinite families of~MOP's. A new family of MOP's satisfying second order differential equations whose coefficients do not behave asymptotically as the identity
matrix was found in~\cite{borrego}; see also~\cite{cantero}.
We have studied~\cite{cum,carlos2} matrix extensions of the generalized polynomials studied in~\cite{adler-van-moerbeke,adler-vanmoerbeke-2}.
Recently, in~\cite{nuevo}, the Christoffel transformation to m,atrix orthogonal polynomials in the real line (MOPRL)
have extended to obtaining a new matrix Christoffel formula, and in~\cite{nuevo2,nuevo3} more general transformations --of Geronimus and Uvarov type-- where also considered.

It was 26 years ago, on 1992, when Fokas, Its and Kitaev, in the context of 2D quantum gravity, discovered that certain Riemann-Hilbert problem was solved in terms of orthogonal polynomials in the real line (OPRL),~\cite{FIK}. Namely, it was found that the solution of a $2\times 2$
Riemann--Hilbert problem can be expressed in terms of orthogonal polynomials in the real line and its Cauchy transforms. Later, Deift and Zhou combined these ideas with a non-linear steepest descent analysis in a series of papers~\cite{deift1,deift2,deift3,deift4} which was the seed for a large activity in the field. To mention just a few relevant results let us cite the study of strong asymptotic with applications in random matrix theory,~\cite{deift1,deift5}, the analysis of determinantal point processes~\cite{daems2,daems3,kuijlaars2,kuijlaars3}, orthogonal Laurent polynomials~\cite{McLaughlin1, McLaughlin2} and Painlevé equations~\cite{dai,kuijlaars4}. 

The study of equations for the recursion coefficients for OPRL or orthogonal polynomials in the unit circle   constitutes a subject of current interest. The question of how the form of the weight and its properties, for example to satisfy a Pearson type equation,
translates to the recursion coefficients has been treated in several places, for a review see~\cite{VAssche}. In~1976, Freud~\cite{freud0} studied weights in $\R$ of exponential variation $w(x)=|x|^\rho\exp(-|x|^m)$, $\rho>-1$ and $m>0$. For $m=2,4,6$ he constructed relations among them as well as determined its asymptotic behavior. However, Freud did not found the role of the discrete Painlevé I, that was discovered later by Magnus~\cite{magnus}. For the unit circle and a weight of the form $w(\theta)=\exp(k\cos\theta)$, $k\in\mathbb R$, Periwal and Shevitz~\cite{Periwal0,Periwal}, in the context of matrix models, found the discrete Painlevé~II equation for the recursion relations of the corresponding orthogonal polynomials. This result was rediscovered latter and connected with the Painlevé III equation~\cite{hisakado}.
In~\cite{baik0} the discrete Painlevé II was found using the Riemann--Hilbert problem given in~\cite{baik}, see also~\cite{tracy}. For a nice account of the relation of these discrete Painlevé equations and integrable systems see~\cite{cresswell}, and for a survey on the subject of differential and discrete Painlev\'e equations cf.~\cite{Clarkson_1}. We also mention the recent paper~\cite{clarkson} where
a discussion on the relationship between the recurrence coefficients of orthogonal polynomials with respect to a semiclassical Laguerre weight and classical solutions of the fourth Painlevé equation can be found. Also, in~\cite{Clarkson_2} the solution of the discrete alternate Painlev\'e equations is presented in terms of the Airy function.

In~\cite{CM} the Riemann--Hilbert problem
for this matrix situation and the appearance of non-Abelian discrete versions of Painlevé I were explored, showing singularity confinement~\cite{CMT}. The
singularity analysis for a matrix discrete version of the Painlevé I
equation was performed. It was found that the singularity confinement holds
generically, i.e. in the whole space of parameters except possibly for
algebraic subvarieties. The situation was considered in~\cite{Cassatella_3} for the matrix extension of the Szeg\H{o} polynomials in the unit circle and corresponding non-Abelian versions discrete Painlevé II equations 
For an alternative discussion of the use of Riemann--Hilbert problem for MOPRL see~\cite{GIM}.

Let us mention that in~\cite{miranian,miranian2} and~\cite{Cafasso} the MOP are expressed in terms of Schur complements that play the role of determinants in the standard scalar case. In~\cite{Cafasso} an study of matrix Szegő polynomials and the relation with a non Abelian Ablowitz--Ladik lattice is carried out, and in~\cite{carlos} the CMV ordering is applied to study orthogonal Laurent polynomials in the circle.

In this work we obtain Sylvester systems of differential equations for the orthogonal polynomials and its second kind functions, directly from a Riemann--Hilbert problem, with jumps supported on appropriate curves in the complex plane. The differential properties for the weight function are fundamental. In this case we consider a Sylvester type differential Pearson equation for the matrix of weights.
We also study whenever the orthogonal polynomials and its second kind functions are solutions of a second order linear differential operators with matrix eigenvalues. This is done by stating an appropriate boundary value problem for the matrix of weights.
In particular, special attention is paid to non-Abelian Hermite biorthogonal polynomials in the real line, understood as those whose matrix of weights is a solution of a Sylvester type Pearson equation with given first order matrix polynomials coefficients.

Several applications are given, in order of increasing complexity, as well. First, we return to the non-Abelian Hermite biorthogonal polynomials in the real line, and give nonlinear matrix difference equations for the recurrent coefficients of the non-Abelian Hermite biorthogonal polynomials. 
Next, we consider the orthogonal polynomials and functions of second kind associated with matrix of weights, that satisfy a differential matrix Person equation with degree two polynomials as coefficients. To simplify the discussion, only a left Pearson equation is considered. In this case, the support of the measure belongs to an appropriate branch of an hyperbola, and the recursion coefficients are shown to fulfill a non-Abelian extension of the scalar alternate discrete Painlevé I equation. 
Finally, a discussion is given for the case of degree three polynomials as coefficients in the left Pearson equation characterizing the matrix of weights. However, for simplicity only odd polynomials are allowed. In this case, a new and more general matrix extension of the discrete Painlevé equation is found. To end this study we present a comparison with the results already obtained by several authors in the scalar and matrix cases.

The layout of the paper is as follows.  In \S~\ref{sec:2} we introduce the basic objects and results fundamental to the rest of the work. Then, \S~\ref{sec:3} is devoted to study the interplay between 
fundamental matrices with constant jump and structure formulas. In \S~\ref{sec:4} and~\ref{sec:5} we characterize sequences of orthogonal polynomials whose matrix weight satisfy a Pearson--Sylvester matrix differential equation by means of a Sylvester matrix differential system and a second order differential operator. Finally, in \S~\ref{sec:6} we show how to derive Painlevé equations for the matrix recurrence coefficients of orthogonal polynomial sequences associated with matrix weight functions of ``exponential'' type.

\section{Riemann--Hilbert problem for Matrix Biorthogonal Polynomials} \label{sec:2}

\subsection{Matrix biorthogonal polynomials }
Let 
\begin{gather*} 
 W = \begin{bmatrix} 
 W^{(1,1)} & \cdots & W^{(1,N)} \\
\vdots & \ddots & \vdots \\
 W^{(N,1)} & \cdots & W^{(N,N)} 
\end{bmatrix}\in \C^{ N\times N}
\end{gather*} 
be a $N \times N$ matrix of weights with support on a smooth oriented non self-intersecting curve~$\gamma$ in the complex plane $\C$, i.e.
 $ W^{(j,k)} $ 
is, for each $j,k \in \{ 1, \ldots , N \}$, a complex weight with support on $\gamma$. 
We define the \emph{moment of order $n $} associated with $ W$ as 
\begin{align*} %\label{eq:momentos}
 W_n &= \frac{1} {2\pi\ii} \int_\gamma z^n W (z) \d z, & n& \in \N 
 :
 =\{0,1,2,\dots\}.
\end{align*}
We say that $ W$ is \emph{regular} if
 $\det \big[ W_{j+k} \big]_{j,k=0, \ldots n}\not = 0 $, $n \in \mathbb N $.
In this way, we define a \emph{sequence of matrix monic polynomials},
 $\big\{ {P}_n^{\mathsf L} (z)\big\}_{n\in\N} $, \emph{left orthogonal} and \emph{right orthogonal}, $\big\{ P_n^{\mathsf R}(z)\big\}_{n\in\N} $ with respect to a regular matrix measure $ W$, by the conditions,
\begin{align} \label{eq:ortogonalidadL}
 \frac 1 {2\pi\ii} \int_\gamma {P}_n^{\mathsf L} (z) W (z) z^k \d z
 = \delta_{n,k} C_n^{-1} , \\
 \label{eq:ortogonalidadR}
 \frac 1 {2\pi\ii} \int_\gamma z^k W (z) {P}_n^{\mathsf R} (z) \d z
 = \delta_{n,k} C_n^{-1} , 
\end{align}
for $k = 0, 1, \ldots , n $ and $n \in \mathbb N$,
where $C_n $ is an nonsingular matrix.

Notice that neither the matrix of weights is requested to be Hermitian 
nor the curve~$\gamma$ to be the real line, i.e., we are dealing, in principle with nonstandard orthogonality and, consequently, with biorthogonal matrix polynomials instead of orthogonal matrix polynomials.

The matrix of weights induce a sesquilinear form in the set of matrix 
%orthogonal
polynomials $\C^{N\times N}[z]$ given by
\begin{align}\label{eq:sesquilinear}
\langle P,Q\rangle_ W :=\frac 1 {2\pi\ii} \int_\gamma {P} (z) W (z) Q(z) \d z.
\end{align}
Moreover, we say that $\big\{ P_n^{\mathsf L}(z)\big\}_{n\in\N}$ and $\big\{ P_n^{\mathsf R}(z)\big\}_{n\in\N}$ are biorthogonal with respect to a matrix weight functions $ W$ if
\begin{align} \label{eq:biorthogonality}
\big\langle P_n^{\mathsf L}, {P}_m^{\mathsf R} \big\rangle_ W
& = \delta_{n,m} C_n^{-1}, &\ n , m &\in \mathbb N .
\end{align}
%From hereon we concentrate on left orthogonality
 As the polynomials are chosen to be monic, we can write
\begin{align*}
 {P}_n^{\mathsf L} (z) = I z^n + p_{\mathsf L,n}^1 z^{n-1} + p_{\mathsf L,n}^2 z^{n-2} + \cdots + p_{\mathsf L,n}^n , 
 \\
 {P}_n^{\mathsf R} (z) = I z^n + p_{\mathsf R,n}^1 z^{n-1} + p_{\mathsf R,n}^2 z^{n-2} + \cdots + p_{\mathsf R,n}^n , 
\end{align*}
with matrix coefficients $p_{\mathsf L, n}^k , p_{\mathsf R, n}^k \in 
\C^{N\times N} $, $k = 0, \ldots, n $ and $n \in \mathbb N $ (imposing that $p_{\mathsf L,n}^0 = p_{\mathsf R,n}^0=I $, $n \in \mathbb N $). Here $I\in\C^{N\times N}$ denotes the identity matrix.

\subsection{Three term relations}
From \eqref{eq:ortogonalidadL} we deduce that the Fourier coefficients of the expansion
\begin{gather*}
z {P}^{\mathsf L}_n (z) = \sum_{k=0}^{n+1} \ell_{\mathsf L,k}^n {P}^{\mathsf L}_k (z) ,
\end{gather*}
are given by $\ell_{\mathsf L,k}^n = 0_N $, $k = 0 , 1, \ldots , n-2 $ (here we denote the zero matrix by $0_N$),
 $\ell_{\mathsf L,n-1}^n = C_n^{-1} C_{n-1} $ (is a direct consequence of orthogonality conditions), $\ell_{\mathsf L,n+1}^n = I \,$ (as $ {P}^{\mathsf L}_n (z) $ are monic polynomials) and 
 $\ell_{\mathsf L,n}^n = 
p_{\mathsf L,n}^1 - p_{\mathsf L,n+1}^1= 
 :
\beta^\mathsf L_n $
(by comparison of the coefficients, assuming $C_0 = I$). 

Hence, assuming the orthogonality relations~\eqref{eq:ortogonalidadL}, we conclude that the sequence of monic polynomials $\big\{ {P}^{\mathsf L}_n (z)\big\}_{n\in\N} $ is defined by the three term recurrence relations
\begin{align} \label{eq:ttrr}
z {P}^{\mathsf L}_n (z) &= {P}^{\mathsf L}_{n+1} (z) + \beta^{\mathsf L}_n {P}^{\mathsf L}_{n} (z) + \gamma^{\mathsf L}_n {P}^{\mathsf L}_{n-1} (z),& \ n &\in \mathbb N, 
\end{align}
with recursion coefficients 
\begin{align*}
\beta^\mathsf L_n& :=
p_{\mathsf L,n}^1 - p_{\mathsf L,n+1}^1, &
\gamma^{\mathsf L}_n 
 &:= C_n^{-1} C_{n-1},
\end{align*}
 with initial conditions, $ {P}^{\mathsf L}_{-1} = 0_N$ and 
 $ {P}^{\mathsf L}_{0} = I $.

Any sequence of monic matrix polynomials, $\big\{ P^{\mathsf R}_{n}(z) \big\}_{n\in\N} $, with $\deg P^{\mathsf R}_n= n $, biorthogonal with respect to $\big\{ P_n^{\mathsf L} (z)\big\}_{n\in\N}$ and $ W(z)$, i.e.~\eqref {eq:biorthogonality}
is fulfilled, also satisfies a three term relation. To prove this we compute the Fourier coefficients of $zP^{\mathsf R}_m (z)$ in the expansion
\begin{align*}
zP^{\mathsf R}_n (z) &= \sum_{k= 0}^{n+1} P^{\mathsf R}_k (z) \ell^n_{\mathsf R,k}, &
 \ell_{\mathsf R,k}^n &=
\frac 1 {2 \pi \ii} \int_\gamma zP^{\mathsf L}_k (z) W(z) P^{\mathsf R}_{n} (z) \d z .
\end{align*}
From~\eqref{eq:ortogonalidadL} we have
$ \ell_{\mathsf R,n+1}^n = I $, $\ell_{\mathsf R,n}^n = C_n \beta^{\mathsf L}_n C_n^{-1} $, $\ell_{\mathsf R,n-1}^n = C_{n-1}C_n^{-1} $, and $\ell_{\mathsf R,k}^n = 0_N$, $k = 0, \dots, n-2$, i.e.
the sequence of monic polynomials $\big\{ P^{\mathsf R}_n(z) \big\}_{n\in\N}$
satisfies
\begin{align} \label{eq:rightttrr}
P^{\mathsf R}_{-1}& = 0_N, &
 P^{\mathsf R}_{0} &= I , & 
 z P^{\mathsf R}_n (z) &= P^{\mathsf R}_{n+1} (z) + P^{\mathsf R}_{n} (z) \beta^{\mathsf R}_n + P^{\mathsf R}_{n-1} (z) \gamma^{\mathsf R}_n ,& n &\in\N,
\end{align}
where 
\begin{align*}
\beta_n^{\mathsf R} &:= C_n \beta^{\mathsf L}_n C_n^{-1}, &\gamma^{\mathsf R}_n &:= C_n \gamma^{\mathsf L}_nC_n^{-1} = C_{n-1} C_n^{-1},
\end{align*}
and the orthogonality conditions~\eqref{eq:ortogonalidadR} are satisfied.

\subsection{Second kind functions}

We define the \emph{sequence of second kind matrix functions} by
\begin{align} \label{eq:secondkind}
Q^{\mathsf L}_n (z) &
 :
= \frac1{ 2 \pi \ii} \int_\gamma \frac{P^{\mathsf L}_n (z')}{z'-z} { W (z')} \d z^\prime , 
 \\ 
\label{eq:right_secondkind}
{Q}_n^{\mathsf R} (z) &
 :
= \frac{1}{2\pi \ii} \int_\gamma W (z') \frac{P^{\mathsf R}_{n} (z')}{z'-z} \d z' ,
\end{align} 
for $ n \in \mathbb N$.
From the orthogonality conditions \eqref{eq:ortogonalidadL} and \eqref{eq:ortogonalidadR} we have, for all $n \in \N$, the following asymptotic expansion near infinity for the sequence of functions of the second~kind 
\begin{align} 
 \label{eq:secondkindlaurent}
Q^{\mathsf L}_n (z) 
& = - C_n^{-1} \big( I z^{-n-1} + q_{\mathsf L,n}^1 z^{-n-2} + \cdots \big), %&
   \\
  \label{eq:secondkindlaurentR}
Q^{\mathsf R}_n (z) 
& = - \big( I z^{-n-1} + q_{\mathsf R,n}^1 z^{-n-2} + \cdots \big)C_n^{-1} . 
\end{align}
Assuming that the measures $ W^{(j,k)} $, $j,k \in \{ 1, \ldots , N \} $ are Hölder continuous we obtain, by the Plemelj's formula applied to~\eqref{eq:secondkind} and~\eqref{eq:right_secondkind}, the following fundamental jump identities
\begin{align} \label{eq:inverseformula}
\big( Q^{\mathsf L}_n (z) \big)_+ - \big( Q_n (z)^{\mathsf L} \big)_- &= {P}^{\mathsf L}_n (z) W (z), %&
  \\
\label{eq:inverseformulaR}
\big( Q^{\mathsf R}_n (z) \big)_+ - \big( Q^{\mathsf R}_n (z) \big)_- &= W (z){P}^{\mathsf R}_n (z) ,
 %&
\end{align}
$z\in\gamma$, where, 
$ \big( f(z) \big)_{\pm} = \lim\limits_{\epsilon \to 0^{\pm}} f(z + i \epsilon ) $; 
here $\pm$ indicates the the positive/negative region according to the orientation of the curve $\gamma$.

Now, multiplying this equation on the right by $ W$ and integrating we get, using the definition~\eqref{eq:secondkind} of $\big\{ Q^{\mathsf L}_n (z)\big\}_{n\in\N}$, that
\begin{gather*}
\frac{1} {2\pi\ii} \int_\gamma \frac{z' {P}^{\mathsf L}_n (z')} {z'-z} { W (z')} \d z' = Q^{\mathsf L}_{n+1} (z) + \beta^{\mathsf L}_n Q^{\mathsf L}_{n} (z) + C_n^{-1} C_{n-1} Q^{\mathsf L}_{n-1} (z).
\end{gather*}
As $\frac{z'}{z'-z} = 1 + \frac{z}{z'-z} $, from the orthogonality conditons \eqref{eq:ortogonalidadL} we conclude that
\begin{align*}
z Q^{\mathsf L}_n (z) & = Q^{\mathsf L}_{n+1} (z) 
+ \beta^{\mathsf L}_n Q^{\mathsf L}_{n} (z) + C_n^{-1} C_{n-1} Q^{\mathsf L}_{n-1} (z) , & n \in \mathbb N ,
\end{align*}
with initial conditions 
\begin{gather*} 
Q^{\mathsf L}_{-1} (z) = Q^{\mathsf R}_{-1} (z) = - C_{-1}^{-1} 
\quad \mbox{and} \quad
Q^{\mathsf L}_0 (z)=Q^{\mathsf R}_0(z)=S_ W(z):= \frac{1}{2\pi\ii}\int_{\gamma}\frac{ W(z')}{z'-z}\d z' , 
\end{gather*}
where $S_ W(z)$ is the 
Stieltjes--Markov 
transformation of the matrix of weights $ W$, which is a complex measure of orthogonality for $\big\{ {P}^{\mathsf L}_n(z) \big\}_{n\in\N} $ --direct consequence of Fubini theorem and Cauchy integral formula.
It can be seen that
\begin{gather*}
 {P}^{\mathsf L}_n (z) Q_0 (z) = -\frac{1} {2\pi\ii} \int \frac{ {P}^{\mathsf L}_n (z') - {P}^{\mathsf L}_n (z)} {z'-z} { W (z') } \operatorname{d} z' +
\frac{1} {2\pi\ii} \int \frac{{P}^{\mathsf L}_n (z')}{z'-z} { W (z') } \d z',
\end{gather*}
i.e. we have the Hermite--Padé formula for the left orthogonal polynomials,
\begin{align*}
 {P}^{\mathsf L}_n (z) S_ W(z) + {P}_{n-1}^{\mathsf L,(1)} (z)& =Q^{\mathsf L}_n (z) , & n \in \mathbb N ,
\end{align*} 
where
\begin{align*} 
{P}_{n-1}^{\mathsf L,(1)} (z) & = \frac{1} {2\pi\ii} \int \frac{ P^{\mathsf L}_n (z') - P^{\mathsf L}_n (z)} {z'-z} { W (z') } \operatorname{d} w , & n \in \mathbb N ,
\end{align*} 
is a polynomial of degree at most $n-1 $ said to be the \emph{first kind associated polynomial with respect to 
$\big\{ P^{\mathsf L}_n (z) \big\}_{n\in\N} $ and $ W(z)$}.
Similarly, for the right situation we have the associated 
\begin{align*}
P_n^{\mathsf R,(1)} (z)& = \frac{1}{2\pi i} \int_\gamma W (z')\frac{ P^{\mathsf R}_{n+1} (z') - P^{\mathsf R}_{n+1} (z) }{z'-z}
\d w, & n &\in\N ,
\end{align*}
and the corresponding Hermite--Padé formula for the right orthogonal polynomials,
\begin{align*}
 S_ W(z){P}^{\mathsf R}_n (z) + {P}_{n-1}^{\mathsf R,(1)} (z)& = Q^{\mathsf R}_n (z) & n \in \mathbb N .
\end{align*}

\subsection{Reductions: from biorthogonality to orthogonality}

We consider two possible reductions for the matrix of weights, the symmetric reduction and the Hermitian reduction.

\begin{enumerate}
 \item A matrix of weights $W(z)$ with support on $\gamma$ is said to be symmetric if
\begin{align*}
(W(z))^\top&=W(z),& z&\in\gamma.
\end{align*}
\item A matrix of weights $W(x)$ with support on $\R$ is said to be Hermitian if
\begin{align*}
(W(x))^\dagger&=W(x),& x&\in\R.
\end{align*}
 \end{enumerate}

These two reductions leads to orthogonal polynomials, as the two biorthogonal families are identified; i.e., for the symmetric case
\begin{align*}
P_n^\mathsf R (z)&=\big(P_n^\mathsf L (z)\big)^\top, &
Q_n^\mathsf R (z)&=\big(Q_n^\mathsf L (z)\big)^\top,& z&\in\C,
\end{align*}
and for the Hermitian case, with $\gamma=\R$,
\begin{align*}
P_n^\mathsf R (z)&=\big(P_n^\mathsf L (\bar z)\big)^\dagger, &
Q_n^\mathsf R (z)&=\big(Q_n^\mathsf L (\bar z)\big)^\dagger,& z&\in\C.
\end{align*}

In both cases biorthogonality collapses into orthogonality, that
for the symmetric case reads~as
\begin{align*} 
 \frac{1}{2\pi\ii}\int_{\gamma} P_n(z) W(z)\big({P}_m(z)\big)^\top\d z
 & = \delta_{n,m} C_n^{-1}, &\ n , m &\in \mathbb N,
\end{align*}
while for the Hermitian case can be written as follows
\begin{align*} 
 \frac{1}{2\pi\ii}\int_{\R} P_n(
x) W(x)\big({P}_m(x)\big)^\dagger\d x
 & = \delta_{n,m} C_n^{-1}, &\ n , m &\in \mathbb N,
\end{align*}
where $P_n=P_n^\mathsf L $.

\subsection{Fundamental and transfer matrices vs Riemann--Hilbert problems}

We can summarize the left three term relation as follows
\begin{gather*}
\begin{bmatrix}
{P}^{\mathsf L}_{n+1} (z) & Q^{\mathsf L}_{n+1} (z) \\[.05cm]
-C_{n} {P}^{\mathsf L}_n (z) & -C_{n} Q^{\mathsf L}_n (z)
\end{bmatrix}
=
\begin{bmatrix}
z I - \beta^{\mathsf L}_n & C_{n}^{-1} \\[.05cm]
-C_{n} & {0}_N
\end{bmatrix}
\begin{bmatrix}
{P}^{\mathsf L}_{n} (z) & Q^{\mathsf L}_{n} (z) \\[.05cm]
-C_{n-1} {P}^{\mathsf L}_{n-1} (z) & -C_{n-1} Q^{\mathsf L}_{n-1} (z)
\end{bmatrix};
\end{gather*} 
and 
\begin{gather*}
\begin{bmatrix}
{P}_{n}^{{\mathsf L},(1)} (z) \\[.05cm]
-C_{n} {P}_{n-1}^{{\mathsf L},(1)} (z)
\end{bmatrix}
=
\begin{bmatrix}
z I - \beta^{\mathsf L}_n & C_{n}^{-1} \\[.05cm]
-C_{n} & {0}_N
\end{bmatrix}
\begin{bmatrix}
{P}_{n-1}^{{\mathsf L},(1)} (z) \\[.05cm]
-C_{n-1} {P}_{n-2}^{{\mathsf L},(1)} (z)
\end{bmatrix}.
\end{gather*}
In terms of the left \emph{fundamental matrix} $Y_n^{\mathsf L}(z)$ and the left \emph{transfer matrix} $T_n^{\mathsf L}(z)$,
\begin{align*} %\label{eq:representaYR}
 Y^{\mathsf L}_n (z) & 
 :
= 
\begin{bmatrix}
 {P}^{\mathsf L}_{n} (z) & Q^{\mathsf L}_{n} (z) \\[.05cm]
-C_{n-1} {P}^{\mathsf L}_{n-1} (z) & -C_{n-1} Q^{\mathsf L}_{n-1} (z)
\end{bmatrix}
, &
 T^{\mathsf L}_n (z)&
 :
= 
\begin{bmatrix}
z I - \beta_n^{\mathsf L}& C_{n}^{-1} \\[.05cm]
-C_{n} & {0}_N
\end{bmatrix},
\end{align*}
we rewrite the above identities as follows
\begin{align*}
Y_{n+1}^{\mathsf L} (z) &= T^{\mathsf L}_n (z) Y^{\mathsf L}_n (z) , & n &\in \mathbb N .
\end{align*}
From these we see that
 $\det Y^{\mathsf L}_n (z) = \det Y^{\mathsf L}_{0} (z) = 1 $, as $\det T^{\mathsf L}_n = 1 $, $n \in \mathbb N $.

For the right orthogonality, we similarly obtain from~\eqref{eq:rightttrr} that
\begin{align*}
\begin{bmatrix}
P^{\mathsf R}_{n+1} (z) & - P^{\mathsf R}_{n} (z) C_n \\[.05cm]
{Q}^{\mathsf R}_{n+1} (z) & - {Q}^{\mathsf R}_n (z) C_n
\end{bmatrix}
=
\begin{bmatrix}
P^{\mathsf R}_{n} (z) & - P^{\mathsf R}_{n-1} (z) C_{n-1} \\[.05cm]
{Q}^{\mathsf R}_{n} (z) & - {Q}^{\mathsf R}_{n-1} (z) C_{n-1}
\end{bmatrix}
\begin{bmatrix}
z I -\beta_n^{\mathsf R} & - C_{n} \\[.05cm]
C_{n}^{-1} & {0}_N
\end{bmatrix}
\end{align*}
and also
\begin{gather*}
\begin{bmatrix}
P^{\mathsf R,(1)}_{n} (z) & - P^{\mathsf R,(1)}_{n-1}(z) C_{n}
\end{bmatrix}
=
\begin{bmatrix}
P^{\mathsf R,(1)}_{n-1} (z) & - P^{\mathsf R,(1)}_{n-2}(z) C_{n}
\end{bmatrix}
\begin{bmatrix}
z I - \beta^{\mathsf R}_n & -C_{n} \\[.05cm]
C_{n}^{-1} & 0_N
\end{bmatrix}
\end{gather*}
as we have the Hermite-Pad\'e formula for the right orthogonal polynomials,
\begin{gather*}
{Q}^{\mathsf R}_0 (z) \, P^{\mathsf R}_m (z) + {P}_{m-1}^{{\mathsf R},(1)} (z) = Q^{\mathsf R}_{m} (z)
\, .
\end{gather*}
Taking the right versions of \emph{fundamental matrix} $Y_n^{\mathsf R}(z)$ and \emph{transfer matrix} $T_n^{\mathsf R}(z)$,
\begin{align*} %\label{eq:right_fundamentaYR}
{Y}^{\mathsf R}_n (z) &
 : 
= 
\begin{bmatrix}
P^{\mathsf R}_{n} (z) & - P^{\mathsf R}_{n-1} (z) C_{n-1} \\[.05cm]
{Q}^{\mathsf R}_{n} (z) & - {Q}^{\mathsf R}_{n-1} (z) C_{n-1}
\end{bmatrix}, &
{T}^{\mathsf R}_m (z) &
 :
=
\begin{bmatrix}
z I - \beta^{\mathsf R}_n & - C_{n} \\[.05cm]
C_{n}^{-1} & 0_N
\end{bmatrix},
\end{align*}
we see that
$\det {Y}^{\mathsf R}_n (z) = \det \, Y^{\mathsf R}_{0} (z) = 1$, because
$\det {T}^{\mathsf R}_n= 1$, $n \in \N$.

Note that,
\begin{align*}
T^{\mathsf R}_n(z)
 & =
\begin{bmatrix}
C_n & 0_N \\ 
0_N & - C_{n}^{-1}
\end{bmatrix}
{T}^{\mathsf L}_n (z) 
\begin{bmatrix}
C_n & 0_N \\ 
0_N & - C_{n}^{-1}
\end{bmatrix}^{-1} , & n \in \mathbb N .
\end{align*}
As a conclusion we arrive to the following left Riemann--Hilbert problem.
\begin{teo} \label{teo:LRHP}
The matrix function
 $ Y_n^{\mathsf L}(z) $ 
is, for each $n \in \N$, the unique solution of the Riemann--Hilbert problem; which consists in the determination of a $2N \times 2 N$ complex matrix function such~that:

\hangindent=.9cm \hangafter=1
{\noindent}$\phantom{ol}${\rm (RH1):} $ Y_n^{\mathsf L} (z)$ is holomorphic in $\C \setminus \gamma $;

\hangindent=.9cm \hangafter=1
{\noindent}$\phantom{ol}${\rm (RH2):} has the following asymptotic behavior near infinity,
\begin{gather*} 
Y_n^{\mathsf L} (z) = \big( I + \operatorname{O} (z^{-1}) \big)
\begin{bmatrix}
I z^n & {0}_N \\
{0}_N & I z^{-n} 
\end{bmatrix} ;
\end{gather*}

\hangindent=.9cm \hangafter=1
{\noindent}$\phantom{ol}${\rm (RH3):} satisfies the jump condition
\begin{align*} 
 \big( Y^{\mathsf L}_n (z) \big)_+ &= \big( Y^{\mathsf L}_n (z) \big)_- \,
\begin{bmatrix}
I & W (z) \\ 
{0}_N & I 
\end{bmatrix}, &z &\in \gamma .
\end{align*} 
\end{teo}
As well as its right version.
\begin{teo} \label{teo:RRHP}
The matrix function $ Y_n^{\mathsf R}(z) $ is, for each $n \in \N$, the unique solution of the Riemann--Hilbert problem; which consists in the determination of a $2N \times 2 N$ complex matrix function such~that:
 
\hangindent=.9cm \hangafter=1
{\noindent}$\phantom{ol}${\rm (RH1):} $ Y_n^{\mathsf R} (z) $ is holomorphic in $\C \setminus \gamma $;
 
\hangindent=.9cm \hangafter=1
{\noindent}$\phantom{ol}${\rm (RH2):} has the following asymptotic behavior near infinity,
\begin{gather*} 
 Y_n^{\mathsf R} (z) = 
 \begin{bmatrix}
 I z^n & {0}_N \\ 
 {0}_N & I z^{-n} 
 \end{bmatrix}\big( I + \operatorname{O} (z^{-1}) \big);
\end{gather*}
 
\hangindent=.9cm \hangafter=1
{\noindent}$\phantom{ol}${\rm (RH3):} satisfies the jump condition
\begin{align*} 
 \big( Y^{\mathsf R}_n (z) \big)_+ &=
 \begin{bmatrix}
 I & {0}_N \\ 
  W (z) & I 
 \end{bmatrix} \big( Y^{\mathsf R}_n (z) \big)_- , &z &\in \gamma .
\end{align*} 
\end{teo}

\begin{rem}
Conditions (RH2) and (RH3) are direct consequences of the representation of the second kind functions~\eqref{eq:secondkindlaurent},~\eqref{eq:secondkindlaurentR} and the inverse formulas~\eqref{eq:inverseformula},~\eqref{eq:inverseformulaR}, respectively.
\end{rem}
\begin{rem}
 For the symmetric and Hermitian reductions these two Riemann--Hilbert problems are the same and for the fundamental matrices we have
\begin{align*}
 Y^\mathsf R_n(z)&=\big(Y_n^\mathsf L(z)\big)^\top, & \text{symmetric case,}\\
 Y^\mathsf R_n(z)&=\big(Y_n^\mathsf L(\bar z)\big)^\dagger, & \text{Hermitian case}.
\end{align*}
In both cases, we will use the notation
\begin{align*}
Y_n(z):=Y_n^\mathsf L(z).
\end{align*}
\end{rem}

We define the family of \emph{normalized left fundamental matrices} 
$\big\{ S^{\mathsf L}_n(z) \big\}_{n\in\N} $ associated with 
$\big\{ Y^{\mathsf L}_n (z)\big\}_{n\in\N} $ by means of
\begin{align*} %\label{eq:fundamentalmatrix}
 S^{\mathsf L}_n (z) & 
 :
 = Y^{\mathsf L}_n (z) 
 \begin{bmatrix} 
I z^{-n} & 0_N \\ 
0_N & I z^{n} 
\end{bmatrix} ,& n &\in \N.
\end{align*}
Taking into account the representation of $\big\{ P^{\mathsf L}_n (z) \big\}_{n\in\N}$ and $\big\{ Q^{\mathsf L}_n (z)\big\}_{n\in\N}$ in~\eqref{eq:ttrr}, we arrive to the asymptotic representation for the normalized fundamental matrices
\begin{align*}%\label{eq:sn}
 S^{\mathsf L}_n (z)= I 
 + 
\begin{bmatrix} 
 p_{{\mathsf L},n}^1 & - C_{n}^{-1} \\[.05cm]
 - C_{n-1} & q_{{\mathsf L},n-1}^1 
\end{bmatrix} z^{-1}
 + 
 \begin{bmatrix} 
 p_{{\mathsf L},n}^2 & - C_{n}^{-1} q_{{\mathsf L},n}^{1} \\[.05cm]
 - C_{n-1} p_{{\mathsf L},n-1}^1 & q_{{\mathsf L},n-1}^2
\end{bmatrix} z^{-2}
%+ 
%\begin{bmatrix} 
% p_{n}^3 & C_{n}^{-1} q_n^{2} \\[.05cm] 
% C_{n-1} p_{n-1}^2 & q_{n-1}^3
%\end{bmatrix}z^{-3} 
+ O(z^{-3}) ,
\end{align*}
for $|z|\to\infty$, where 
\begin{align*}
p_{{\mathsf L},n}^1 - p_{{\mathsf L},n+1}^1&= \beta^{\mathsf L}_n , \\
p_{{\mathsf L},n}^2 - p_{{\mathsf L},n+1}^2 &= \beta^{\mathsf L}_n p_{{\mathsf L},n}^1 + C_n^{-1} C_{n-1} , \\ 
p_{{\mathsf L},n}^3 - p_{{\mathsf L},n+1}^3 &= \beta^{\mathsf L}_n p_{{\mathsf L},n}^2 + C_n^{-1} C_{n-1} p_{{\mathsf L},n-1}^1 ,
\end{align*}
and
\begin{align*}
 q_{{\mathsf L},n}^1 - q_{{\mathsf L},n-1}^1& = \beta^{\mathsf R}_n , \\ 
 q_{{\mathsf L},n}^2 - q_{{\mathsf L},n-1}^2 &= \beta^{\mathsf R}_n q_{{\mathsf L},n}^1 + C_n C_{n+1}^{-1}.
\end{align*}
Observe that we will also have the following asymptotics for $|z|\to\infty$,
\begin{multline*} %\label{eq:sn1}
 \big(
 S^{\mathsf L}_n(z)
 \big)^{-1} =
 I - \begin{bmatrix} 
 p_{{\mathsf L},n}^1 & -C_{n}^{-1} \\[.05cm]
 -C_{n-1} & q_{{\mathsf L},n-1}^1 
\end{bmatrix} z^{-1}
 \\ 
+ \Bigg( \begin{bmatrix} 
 p_{{\mathsf L},n}^1 & -C_{n}^{-1} \\[.05cm] 
 -C_{n-1} & q_{{\mathsf L},n-1}^1 
\end{bmatrix}^2 
- \begin{bmatrix} 
 p_{{\mathsf L},n}^2 & -C_{n}^{-1} q_{{\mathsf L},n}^{1} \\[.05cm] 
 -C_{n-1} p_{{\mathsf L},n-1}^1 & q_{{\mathsf L},n-1}^2
\end{bmatrix} \Bigg)z^{-2} + O(z^{-3}) .
\end{multline*}

For the right version we have \emph{normalized right fundamental matrices} 
$\big\{ S^{\mathsf R}_n(z) \big\}_{n\in\N} $ associated with $\big\{ Y^{\mathsf R}_n (z)\big\}_{n\in\N} $ 
\begin{align*} %\label{eq:Rfundamentalmatrix}
S_n^{\mathsf R} (z) =
\begin{bmatrix} I \, z^{-n} & {0}_N \\
{0}_N & I \, z^{n} \end{bmatrix}
 {Y}^{\mathsf R}_m (z),
\end{align*}
with asymptotic behavior at infinity given by
\begin{align*} %\label{eq:rightsn}
{S}^{\mathsf R}_n(z) = I
+
 \begin{bmatrix}
 p_{\mathsf R,n}^1 & - C_{n-1} \\[.05cm] 
 - C_{n}^{-1} & q_{\mathsf R,n-1}^1
\end{bmatrix} z^{-1}
+
\begin{bmatrix}
 p_{\mathsf R,n}^2 & - p_{\mathsf R,n-1}^1 C_{n-1} \\[.05cm]
 - q_{\mathsf R,n}^{1} C_{n}^{-1} & q_{\mathsf R, n-1}^2
\end{bmatrix} z^{-2}
%+\begin{bmatrix}
% p_{\mathsf R,n}^3 & p_{\mathsf R,n-1}^2 C_{n-1} \\[.05cm]
% q_{\mathsf R,n}^{2} C_{n}^{-1} & q_{\mathsf R,n-1}^3
%\end{bmatrix}z^{-3}
+ O(z^{-3}) ,
\end{align*}
for $|z|\to\infty$, 
and the asymptotics for the inverse matrix is 
\begin{multline*} %\label{eq:rightsn1}
\big({S}^{\mathsf R}_n(z)\big)^{-1} = I
-
 \begin{bmatrix}
p_{\mathsf R,n}^1 & -C_{n-1} \\[.05cm]
-C_{n}^{-1} & q_{\mathsf R,n-1}^1
\end{bmatrix} z^{-1}
 \\ +
\Bigg( \begin{bmatrix}
p_{\mathsf R,n}^1 & -C_{n-1} \\[.05cm] 
-C_{n}^{-1} & q_{\mathsf R,n-1}^1
\end{bmatrix}^2 - \begin{bmatrix}
p_{\mathsf R,n}^2 & -p_{\mathsf R,n-1}^1 C_{n-1} \\[.05cm] 
-q_{\mathsf R,n}^{1} C_{n}^{-1} & q_{\mathsf R, n-1}^2
\end{bmatrix} \Bigg)z^{-2}+ O(z^{-3}) .
\end{multline*}
Here
\begin{align*}
p_{{\mathsf R},n}^1 - p_{{\mathsf L},n+1}^1&= \beta^{\mathsf R}_n , \\
p_{{\mathsf R},n}^2 - p_{{\mathsf L},n+1}^2 &= p_{{\mathsf L},n}^1 \beta^{\mathsf R}_n + 
C_{n-1} C_n^{-1} , \\ 
p_{{\mathsf R},n}^3 - p_{{\mathsf R},n+1}^3 &= p_{{\mathsf L},n}^2 \beta^{\mathsf R}_n + p_{{\mathsf L},n-1}^1 C_{n-1} C_n^{-1} ,
\end{align*}
and
\begin{align*}
q_{{\mathsf R},n}^1 - q_{{\mathsf L},n-1}^1& = \beta^{\mathsf L}_n , \\ 
q_{{\mathsf R},n}^2 - q_{{\mathsf L},n-1}^2 &= q_{{\mathsf L},n}^1 \beta^{\mathsf L}_n + C_{n+1}^{-1} C_n .
\end{align*}

%
%
%
%
%we identify them, in terms of sequence of matrix coefficients
%$\big( \beta_n \big) \, $ and $\big( C_n \big) \, $, that appear of the three term recurrence relation for $\big\{ P_n \big\} \, $ and $\big\{ R_n \big\}$, cf.~\eqref{eq:ttrr} and~\eqref{eq:rightttrr}, respectively.
%
%
%
%Taking into account the coefficients of $x^k \, $, $k=0, \ldots , n \, $
%in the recurrence relation for \
%$ P^{\mathsf R}_m := z^m \, I + z^{m-1} \, \widetilde p_{m}^1 + z^{m-2} \, \widetilde p_{m}^2 + \cdots \, $ \ and \
%%\begin{gather*} %\label{eq:rightsecondkindrepresenta}
%${Q}^{\mathsf R}_m :=
%\big( I/z^{m+1} + \widetilde q_{m}^1 /z^{m+2}+ \widetilde q_{m}^2/z^{m+3} + \cdots \big) \, C_m^{-1} \, , \ m \in \mathbb N \, $,
%\ we get,
%\begin{gather*}
%\begin{matrix}
%\widetilde p_{m}^1 - \widetilde p_{m+1}^1 = C_m \, \beta_m \, C_m^{-1} \, ,
%& \widetilde q_{m}^1 - \widetilde q_{m-1}^1 = \beta_m \, , \\
%\widetilde p_{m}^2 - \widetilde p_{m+1}^2 = \widetilde p_m^1 \, C_m \, \beta_m \, C_m^{-1} + C_{m-1} \, C_m^{-1} \, ,
%& \widetilde q_{m}^2 - \widetilde q_{m-1}^2 = \widetilde q_m^1 \, \beta_m + C_{m+1}^{-1} \, C_m \, , \\
%\widetilde p_{m}^3 - \widetilde p_{m+1}^3 = \widetilde p_m^2 \, C_m \, \beta_m \, C_m^{-1} + \widetilde p_m^1 \, C_{m-1} \, C_m^{-1} \, , \, \ldots \, , \ \ \ \
%& \widetilde q_{m}^3 - \widetilde q_{m-1}^3 = \widetilde q_m^2 \, \beta_m +
%\widetilde q_m^1 \, C_{m+1}^{-1} \, C_m \, , \, \ldots \, .
%\end{matrix}
%\end{gather*}
%It is easy to see that $\big\{ S^{\mathsf R}_n \big\} \, $ have the following representation:

\begin{teo} \label{teo:inversa}
Let $ Y_n^{\mathsf L} $ and $ Y_n^{\mathsf R} $ be, for each $n \in \mathbb N$, the unique solutions of the Riemann-Hilbert problems in Theorems~\ref{teo:LRHP} and~\ref{teo:RRHP}, respectively; then
\begin{gather} \label{eq:conexao}
(Y_n^{\mathsf L}(z))^{-1} = 
\begin{bmatrix}
0 & I \\ - I & 0
\end{bmatrix}
Y_n^{\mathsf R} (z) 
\begin{bmatrix}
0 & -I \\ I & 0
\end{bmatrix} , \quad n \in \mathbb N.
\end{gather}
\end{teo}

\begin{proof}
Let us remember that $\big\{ P_n^{\mathsf L} \big\}_{n\in\N}$ satisfy~\eqref{eq:ttrr}, i.e. 
\begin{align*} %\label{eq:ttrr}
z {P}^{\mathsf L}_n (z) &= {P}^{\mathsf L}_{n+1} (z) + \beta^{\mathsf L}_n {P}^{\mathsf L}_{n} (z) + C_n^{-1} C_{n-1} {P}^{\mathsf L}_{n-1} (z),& \ n &\in \mathbb N, 
\end{align*}
with initial conditions $P^{\mathsf L}_{-1} = 0_N $ and
$ P^{\mathsf L}_{0} = I$; and $\big\{ P_n^{\mathsf R} \big\}_{n\in\N}$ satisfies~\eqref{eq:rightttrr}, i.e.
\begin{align*} %\label{eq:rightttrr} 
 t P^{\mathsf R}_n (t) & = P^{\mathsf R}_{n+1} (t) + P^{\mathsf R}_{n} (t) C_n \beta^{\mathsf L}_n C_n^{-1} + P^{\mathsf R}_{n-1} (t) C_{n-1} C_n^{-1} ,& n &\in\N,
\end{align*}
with initial conditions $P^{\mathsf R}_{-1} = 0_N $ and
$ P^{\mathsf R}_{0} = I$. Multiplying the first equation on the left by $P_n^{\mathsf R} (t) C_n$ and the second one on the right by $ C_n P_n^{\mathsf L} (z) $ and summing up, we arrive after applying telescoping rule
\begin{align} \label{eq:LiouvillePP}
(z-t) \sum_{k=0}^n P_k^{\mathsf R}(t) C_k P_k^{\mathsf L} (z) 
 & = P_n^{\mathsf R}(t) C_n P_{n+1}^{\mathsf L} (z)
 - P_{n+1}^{\mathsf R}(t) C_n P_{n}^{\mathsf L} (z) , & n \in \mathbb N ;
\end{align}
hence for $t = z$,
\begin{align} \label{eq:LiouvilleconfluentePP}
P_n^{\mathsf R}(z) C_n P_{n+1}^{\mathsf L} (z)
 & = P_{n+1}^{\mathsf R}(z) C_n P_{n}^{\mathsf L} (z) , & n \in \mathbb N ;
\end{align}
As $\big\{ Q_n^{\mathsf L} \big\}_{n\in\N}$ (respectively, $\big\{ Q_n^{\mathsf R} \big\}_{n\in\N}$) satisfy~\eqref{eq:ttrr} (respectively,~\eqref{eq:rightttrr}), with initial conditions $Q_{-1}^{\mathsf L} = Q_{-1}^{\mathsf R} = - C_{-1}^{-1}$, $Q_{0}^{\mathsf L} = Q_{0}^{\mathsf R} = S_W (z)$, proceeding in the same way with 
$\big\{ Q_n^{\mathsf L} \big\}_{n\in\N}$ and $\big\{ Q_n^{\mathsf R} \big\}_{n\in\N}$ in place of $\big\{ P_n^{\mathsf L} \}_{n\in\N}$ and 
$\big\{ P_n^{\mathsf R} \big\}_{n\in\N}$, respectively, we arrive to
\begin{align} \label{eq:LiouvilleQQ}
(z-t) \sum_{k=0}^n Q_k^{\mathsf R}(t) C_k Q_k^{\mathsf L} (z) 
 & = Q_n^{\mathsf R}(t) C_n Q_{n+1}^{\mathsf L} (z)
 - Q_{n+1}^{\mathsf R}(t) C_n Q_{n}^{\mathsf L} (z) , & n \in \mathbb N ;
\end{align}
hence for $t = z$,
\begin{align} \label{eq:LiouvilleconfluenteQQ}
Q_n^{\mathsf R}(z) C_n Q_{n+1}^{\mathsf L} (z)
 & = Q_{n+1}^{\mathsf R}(z) C_n Q_{n}^{\mathsf L} (z) , & n \in \mathbb N .
\end{align}
Applying the same procedure mixing the $P$'s and the $Q$'s 
we get, for all $n \in \mathbb N$,
\begin{align} \label{eq:LiouvilleQP}
(z-t) \sum_{k=0}^n Q_k^{\mathsf R}(t) C_k P_k^{\mathsf L} (z) 
 & = Q_n^{\mathsf R}(t) C_n P_{n+1}^{\mathsf L} (z)
 - Q_{n+1}^{\mathsf R}(t) C_n P_{n}^{\mathsf L} (z) + I, \\
 \label{eq:LiouvillePQ}
(z-t) \sum_{k=0}^n P_k^{\mathsf R}(t) C_k Q_k^{\mathsf L} (z) 
 & = P_n^{\mathsf R}(t) C_n Q_{n+1}^{\mathsf L} (z)
 - P_{n+1}^{\mathsf R}(t) C_n Q_{n}^{\mathsf L} (z) - I ,
\end{align}
and when $t = z$ we arrive to, for all $n \in \mathbb N$,
\begin{align}\label{eq:LiouvilleconfluenteQP}
Q_{n+1}^{\mathsf R}(z) C_n P_{n}^{\mathsf L} (z) 
 - Q_n^{\mathsf R}(z) C_n P_{n+1}^{\mathsf L} (z) & = I , 
 %& n \in \mathbb N 
   \\
\label{eq:LiouvilleconfluentePQ}
P_n^{\mathsf R}(z) C_n Q_{n+1}^{\mathsf L} (z) 
 - P_{n+1}^{\mathsf R}(z) C_n Q_{n}^{\mathsf L} (z) & = I . 
 %& n \in \mathbb N
\end{align}
Equations~\eqref{eq:LiouvillePP},~\eqref{eq:LiouvilleQQ},~\eqref{eq:LiouvilleQP} and~\eqref{eq:LiouvillePQ} are known in the literature as Christoffel-Darboux 
formulas.
Now, from~\eqref{eq:LiouvilleconfluentePP},~\eqref{eq:LiouvilleconfluenteQQ},~\eqref{eq:LiouvilleconfluenteQP} and~\eqref{eq:LiouvilleconfluentePQ}
we conclude that
\begin{align*}
\begin{bmatrix}
- Q_{n-1}^{\mathsf R} (z) C_{n-1} & - Q_{n}^{\mathsf R} (z) \\
 P_{n-1}^{\mathsf R} (z) C_{n-1} & P_{n}^{\mathsf R} (z) 
\end{bmatrix}
Y_n^{\mathsf L} (z) & = I , & n \in \mathbb N ,
\end{align*}
and as
\begin{align*}
\begin{bmatrix}
- Q_{n-1}^{\mathsf R} (z) C_{n-1} & - Q_{n}^{\mathsf R} (z) \\
 P_{n-1}^{\mathsf R} (z) C_{n-1} & P_{n}^{\mathsf R} (z) 
\end{bmatrix}
 & = 
\begin{bmatrix}
0 & I \\ -I & 0
\end{bmatrix}
Y_n^{\mathsf R} (z) 
\begin{bmatrix}
0 & -I \\ I & 0
\end{bmatrix} , & n \in \mathbb N ,
\end{align*}
we get the desired result.
\end{proof}

\begin{coro}
In the conditions of theorem~\ref{teo:inversa} we have that for all $n \in \mathbb N$,
\begin{align}
Q_{n}^{\mathsf L} (z) P_{n-1}^{\mathsf R} (z)
 - P_{n}^{\mathsf L} (z) Q_{n-1}^{\mathsf R} (z) & = C_{n-1}^{-1} , 
 %& n \in \mathbb N , 
 \\
P_{n-1}^{\mathsf L} (z) Q_{n}^{\mathsf R} (z)
 - Q_{n-1}^{\mathsf L} (z) P_{n}^{\mathsf R} (z) & = C_{n-1}^{-1} , 
 %& n \in \mathbb N , 
 \\
 Q_{n}^{\mathsf L} (z) P_{n}^{\mathsf R} (z)
 - P_{n}^{\mathsf L} (z) Q_{n}^{\mathsf R} (z) & = 0 
 %& n \in \mathbb N 
 .
\end{align}
\end{coro}

\begin{proof}
As we have already prove the matrix
\begin{gather*}
\begin{bmatrix}
- Q_{n-1}^{\mathsf R} (z) C_{n-1} & - Q_{n}^{\mathsf R} (z) \\
 P_{n-1}^{\mathsf R} (z) C_{n-1} & P_{n}^{\mathsf R} (z) 
\end{bmatrix},
\end{gather*}
is the inverse of $Y_n^{\mathsf L} (z)$, i.e.
\begin{gather*}
Y_n^{\mathsf L} (z) \begin{bmatrix}
- Q_{n-1}^{\mathsf R} (z) C_{n-1} & - Q_{n}^{\mathsf R} (z) \\
 P_{n-1}^{\mathsf R} (z) C_{n-1} & P_{n}^{\mathsf R} (z) 
\end{bmatrix} = I;
\end{gather*}
and multiplying the two matrices we get the result.
\end{proof}

\begin{coro}
In the conditions of theorem~\ref{teo:inversa} we have that for all $n \in \mathbb N$,
\begin{gather*}
(S_n^{\mathsf L} (z))^{-1} = I 
 + 
\begin{bmatrix} 
q_{\mathsf R,n-1}^1 & C_n^{-1} \\ 
C_{n-1} & p_{\mathsf R,n}^1
\end{bmatrix} z^{-1}
+ 
\begin{bmatrix} 
q_{\mathsf R,n-1}^2 & q_{\mathsf R,n}^1 C_n^{-1} \\ 
p_{\mathsf R,n-1}^2 C_{n-1} & p_{\mathsf R,n}^2
\end{bmatrix} z^{-2} + \cdots ,
 \\
(S_n^{\mathsf R} (z))^{-1} = I 
+ 
\begin{bmatrix} 
q_{\mathsf L,n-1}^1 & C_{n-1} \\ 
C_{n}^{-1} & p_{\mathsf L,n}^1
\end{bmatrix} z^{-1}
+ 
\begin{bmatrix} 
q_{\mathsf L,n-1}^2 & C_{n-1} p_{\mathsf L,n-1}^1 \\ 
C_{n}^{-1} p_{\mathsf L,n-1}^2 & p_{\mathsf L,n}^2
\end{bmatrix} z^{-2} + \cdots .
\end{gather*}
\end{coro}

\section{Constant jump on the support, structure matrices and zero curvature} \label{sec:3}

So far we have discuss the connection between biorthogonal families of matrix polynomials for a given matrix of weights 
$ W
%(x)
$ and a specific Riemann--Hilbert problem. Now, to derive difference and/or differential equations satisfied by these families of matrix polynomials we will we move to a simpler setting and we will assume that the following~holds

%\begin{enumerate}
 %\item 
 
\hangindent=.9cm \hangafter=1
{\noindent}$\phantom{ol}$i)
The matrix of weights factors out as $ W (z) = W^{\mathsf L} (z) W^{\mathsf R} (z) $, $z\in\gamma$.

 %\item 
\hangindent=.9cm \hangafter=1
{\noindent}$\phantom{ol}$ii)
The factors $ W^{\mathsf L} $ and $ W^{\mathsf R}$ are the restriction to the curve $\gamma$ of matrices of entire 
functions $ W^{\mathsf L} (z)$ and $ W^{\mathsf R}(z) $, $z\in\C$.
 
 %\item 
\hangindent=.9cm \hangafter=1
{\noindent}$\phantom{ol}$iii)
The right logarithmic derivative $h^\mathsf L(z):=\big( W^{\mathsf L} (z) \big)^{\prime} \big( W^{\mathsf L} (z) \big)^{-1} $ and the left logarithmic derivative 
$h^\mathsf L(z):=\big( W^{\mathsf R} (z) \big)^{-1} \big( W^{\mathsf R}(z) \big)^{\prime} $ are also entire functions.
%\end{enumerate}

We underline that for a given matrix of weights $ W(z)$ we will have many possible factorization $ W (z) = W^{\mathsf L} (z) W^{\mathsf R} (z)$. Indeed, if we define an equivalence relation $( W^{\mathsf L}, W^{\mathsf R}) \sim(\widetilde W^{\mathsf L},\widetilde W^{\mathsf R})$ if and only if $ W^{\mathsf L} W^{\mathsf R}=\widetilde W^{\mathsf L}\widetilde W^{\mathsf R}$, then each matrix of weights $ W$ can be though as a class of equivalence, and can be described the orbit 
\begin{gather*}
\Big\{ ( W^{\mathsf L}\phi,\phi^{-1} W^{\mathsf R}), \ \phi(z) \ \text{ is a nonsingular matrix of entire functions} \Big\} \, .
\end{gather*}
% \textcolor{red}{where $\phi(z)$...}

\subsection{Constant jump on the support}

Given assumptions i) and ii), for each factorization~$ W = W^{\mathsf L} W^{\mathsf R}$, we introduce the \emph{constant jump fundamental matrices } which will be instrumental in what follows
\begin{align} \label{eq:zn}
 Z_n^{\mathsf L}(z) 
& :
= Y^{\mathsf L}_n (z) 
\begin{bmatrix} 
 W^{\mathsf L} (z) & 0_N \\ 
0_N & ( W^{\mathsf R} (z))^{-1} 
\end{bmatrix}
% = 
% \begin{bmatrix} P^{\mathsf L}_n (z) W^{\mathsf L} (z)
% & Q^{\mathsf L}_n (z) ( W^{\mathsf R} (z))^{-1} \\ 
%C_{n-1} P^{\mathsf L}_{n-1} (z) W^{\mathsf L} (z) 
% & C_{n-1} Q^{\mathsf L}_{n-1} (z) ( W^{\mathsf R} (z))^{-1} \end{bmatrix}
 , & \\
 \label{eq:zetan}
{Z}^{\mathsf R}_n (z)
& :
=
\begin{bmatrix} 
 W^{\mathsf R} (z) & 0_N \\ 
0_N & ( W^{\mathsf L} (z))^{-1} 
\end{bmatrix}
{Y}^{\mathsf R}_n (z) , & n \in \mathbb N
%= \begin{bmatrix}
% W^{\mathsf R} (z)P^{\mathsf R}_ n(z) & W^{\mathsf R}(z)P^{\mathsf R}_{n-1} (z) 
%C_{n-1} \\[.05cm]
%( W^{\mathsf L} (z))^{-1}Q^{\mathsf R}_n (z) & ( W^{\mathsf L} (z))^{-1}Q^{\mathsf R}_{n-1} (z) C_{n-1} \end{bmatrix}
 .
\end{align}
Taking inverse on~\eqref{eq:zn} and applying~\eqref{eq:conexao} we see that 
$Z_n^{\mathsf R}$ given in~\eqref{eq:zetan} admits the representation
\begin{align} \label{eq:znright}
Z_n^{\mathsf R} (z)
 & = 
\begin{bmatrix}
0 & -I \\ I & 0
\end{bmatrix}
(Z_n^{\mathsf L} (z))^{-1}
\begin{bmatrix}
0 & I \\ -I & 0
\end{bmatrix} , & n \in \mathbb N .
\end{align}

\begin{pro} \label{prop:zn}
For each factorization $ W = W^{\mathsf L} W^{\mathsf R}$, the constant jump fundamental matrices~$Z^{\mathsf L}_n(z) $ and $Z^\mathsf R_n(z)$ are, for each $n \in \N$, characterized by the following properties: 

%\begin{enumerate}
%\item
\hangindent=.9cm \hangafter=1
{\noindent}$\phantom{ol}${\rm i)}
They are holomorphic on $\C \setminus \gamma$.

%\item
\hangindent=.9cm \hangafter=1
{\noindent}$\phantom{ol}${\rm ii)}
%For each $n$, 
We have the following asymptotic behaviors
\begin{align*}
Z^\mathsf L_n(z)&=\big( I + \operatorname{O} (z^{-1}) \big)
\begin{bmatrix}
z^n W^{\mathsf L} (z) & {0_N} \\ 
{0_N} & I z^{-n} ( W^{\mathsf R} (z))^{-1} 
\end{bmatrix}, \\
Z^\mathsf R_n(z)&=\begin{bmatrix}
z^n W^{\mathsf R} (z) & 0_N \\ 
0_N & ( W^{\mathsf R} (z))^{-1} z^{-n}
\end{bmatrix}\big( I + \operatorname{O} (z^{-1}) \big),
\end{align*}
for $|z|\to\infty$.

%\item
\hangindent=.9cm \hangafter=1
{\noindent}$\phantom{ol}${\rm iii)}
They present the following \emph{constant jump condition} on $\gamma$
\begin{align*} 
\big( Z^{\mathsf L}_n (z) \big)_+ &= \big( Z^{\mathsf L}_n (z) \big)_- 
\begin{bmatrix} I & I \\ 
0_N & I 
\end{bmatrix} , &
\big( {Z}^{\mathsf R}_n (z) \big)_+ &=
\begin{bmatrix} 
I & {0}_N \\ 
I & I 
\end{bmatrix} 
\big( {Z}^{\mathsf R}_n (z) \big)_- ,
\end{align*}
for all $z\in\gamma$ in the support on the matrix of weights.
%\end{enumerate}
\end{pro}

\begin{proof}
We only give the proofs for the left case because their right 
ones %are completely analogous
follows from~\eqref{eq:znright}. 

%\begin{enumerate}
% \item
\hangindent=.9cm \hangafter=1
{\noindent}$\phantom{ol}${\rm i)}
 As the $ W^{\mathsf L} (z)$ and $ W^{\mathsf R} (z)$ are matrices of entire functions the holomorphity properties of $Z^{\mathsf L}_n$ is inherit from that of the fundamental matrices $Y_n^{\mathsf L}$.
% \item

\hangindent=.9cm \hangafter=1
{\noindent}$\phantom{ol}${\rm ii)}
It follows from the asymptotic of the fundamental matrices.
% \item

\hangindent=.9cm \hangafter=1
{\noindent}$\phantom{ol}${\rm iii)} 
From the definition of $Z^{\mathsf L}_n (z)$ we have
\begin{gather*}
\big(Z^{\mathsf L}_n (z) \big)_+ = \big( Y^{\mathsf L}_n (z) \big)_+ 
\begin{bmatrix} 
 W^{\mathsf L} (z) & 0_N \\ 
0_N & ( W^{\mathsf R} (z))^{-1} 
\end{bmatrix} \, ,
\end{gather*}
and taking into account Theorem~\ref{teo:LRHP} we arrive to
\begin{gather*}
\big(Z^{\mathsf L}_n (z) \big)_+ = \big( Y^{\mathsf L}_n (z) \big)_-
\begin{bmatrix} 
I & W^{\mathsf L} (z) W^{\mathsf R} (z) \\ 
0_N & I
\end{bmatrix}
\begin{bmatrix} 
 W^{\mathsf L} (z) & 0_N \\ 
0_N & ( W^{\mathsf R} (z))^{-1} 
\end{bmatrix} \, ;
\end{gather*}
now, as
\begin{gather*}
\begin{bmatrix} 
I & W^{\mathsf L} (z) W^{\mathsf R} (z) \\ 
0_N & I
\end{bmatrix}
\begin{bmatrix} 
 W^{\mathsf L} (z) & 0_N \\ 
0_N & ( W^{\mathsf R} (z))^{-1} 
\end{bmatrix}
 =
\begin{bmatrix} 
 W^{\mathsf L} (z) & 0_N \\ 
0_N & ( W^{\mathsf R} (z))^{-1}
\end{bmatrix}
\begin{bmatrix} 
I & I \\ 
0_N & I 
\end{bmatrix} \, ,
\end{gather*}
we get the desired \emph{constant} jump condition for $Z^{\mathsf L}_n (z)$.
% \begin{gather*} 
% \big( Z^{\mathsf L}_n (z) \big)_+ = \big( Z^{\mathsf L}_n (z) \big)_- \left[ \begin{matrix} I & I \\ 0_N & I \end{matrix} \right] .
% \end{gather*}
%\end{enumerate}
\end{proof}

\begin{rem}
 For the symmetric and Hermitian reductions we have 
 \begin{align*}
W^\mathsf L(z)&=\rho(z), & W^\mathsf R(z)&=(\rho(z))^\top, & W(z)&=\rho(z)\big(\rho(z)\big)^\top,&  Z^\mathsf R(z)
&=\big( Z^\mathsf L(z) \big)^\top, & \text{symmetric},\\
W^\mathsf L(z)&=\rho(z), & W^\mathsf R(z)&=(\rho(\bar z))^\dagger, & W&=\rho(z)\big(\rho(\bar z)\big)^\dagger, & Z^\mathsf R(z)
&=\big( Z^\mathsf L(\bar z) \big)^\dagger, & \text{Hermitian}.
 \end{align*}
 In both cases, we will use the notation
 \begin{align*}
 Z_n(z):=Z_n^\mathsf L(z).
 \end{align*}
\end{rem}

\subsection{Structure matrices}

In parallel to the matrices $Z^{\mathsf L}_n(z)$ and $Z^{\mathsf R}_n(z)$, for each factorization $ W = W^{\mathsf L} W^{\mathsf R}$, we introduce what we call \emph{structure matrices} given in terms of the \emph{right derivative} and \emph{left derivative} (logarithmic derivatives), respectively,
\begin{align*} %\label{eq:Mn}
M^{\mathsf L}_n (z) & 
 :
= \big(Z^{\mathsf L}_n(z) \big)^{\prime} \big(Z^{\mathsf L}_{n}(z)\big)^{-1},&
{M}^{\mathsf R}_n (z) & 
 :
= \big ({Z}^{\mathsf R}_{n}(z)\big)^{-1} \big(Z^{\mathsf R}_n(z)\big)^{\prime}.
\end{align*}
It is not difficult to prove that
\begin{align*}
{M}^{\mathsf R}_{n}(z) & = -
\begin{bmatrix}
0 & -I \\ I & 0
\end{bmatrix}
M_n^{\mathsf L} (z)
\begin{bmatrix}
0 & I \\ -I & 0
\end{bmatrix} , & n \in \mathbb N .
\end{align*}

\begin{pro} \label{prop:Mn}
The following properties hold:
%\begin{enumerate}
%\item 

\hangindent=.9cm \hangafter=1
{\noindent}$\phantom{ol}${\rm i)}
The structure matrices $M^{\mathsf L}_n (z) $ and $ M^{\mathsf R}_n (z) $ are, for each $n \in \mathbb N$, matrices of entire functions in the complex~plane.

%\item 
\hangindent=.9cm \hangafter=1
{\noindent}$\phantom{ol}${\rm ii)}
The transfer matrix satisfies
 \begin{align*}
 T^{\mathsf L}_n (z)Z_n^{\mathsf L}(z) & = Z^{\mathsf L}_{n+1}(z) ,& {Z}^{\mathsf R}_n(z){T}^{\mathsf R}_n (z) &=
 {Z}^{\mathsf R}_{n+1}(z), & n \in \mathbb N .
 \end{align*}

\hangindent=.9cm \hangafter=1
{\noindent}$\phantom{ol}${\rm iii)}
%\item 
The zero curvature formulas 
\begin{align} \label{eq:zerocurvature1}
\begin{bmatrix}
I & 0_N\\
0_N & 0_N
\end{bmatrix}%=\big( T_{n}^{\mathsf L}(z) \big)^{\prime} 
 & = M^{\mathsf L}_{n+1} (z) T^{\mathsf L}_n(z) - T^{\mathsf L}_n (z) M^{\mathsf L}_{n} (z),\\
\begin{bmatrix}
I & 0_N \\
0_N & 0_N
\end{bmatrix}%=\big(T^{\mathsf R}_{n} (z)\big)^{\prime} 
 & =
T^{\mathsf R}_n (z) \, M^{\mathsf R}_{n+1} (z)
- M^{\mathsf R}_{n} (z) T^{\mathsf R}_n (z) , 
\label{eq:zerocurvature2}
%\notag
\end{align}
$n \in \mathbb N $, are fulfilled.

%\item 
\hangindent=.9cm \hangafter=1
{\noindent}$\phantom{ol}${\rm iv)}
The\emph{ second order} zero curvature formulas 
\begin{align} \label{eq:zerocurvature3}
%\big( T_{n}^{\mathsf L}(z) \big)^{\prime}
 \begin{bmatrix}
 I & 0_N\\
 0_N & 0_N
 \end{bmatrix}
 M^{\mathsf L}_{n} (z)+ M^{\mathsf L}_{n+1} (z)
 \begin{bmatrix}
 I & 0_N\\
 0_N & 0_N
 \end{bmatrix}
 %\big( T_{n}^{\mathsf L}(z) \big)^{\prime}
&= \big(M^{\mathsf L}_{n+1} (z)\big)^2 
T^{\mathsf L}_n(z) - T^{\mathsf L}_n (z) \big( M^{\mathsf L}_{n} (z)\big)^2,\\
%\big(T^{\mathsf R}_{n} (z)\big)^{\prime}
\begin{bmatrix}
I & 0_N\\
0_N & 0_N
\end{bmatrix}M^{\mathsf R}_{n+1} (z)+M^{\mathsf R}_{n} (z) \begin{bmatrix}
I & 0_N\\
0_N & 0_N
\end{bmatrix}
%\big(T^{\mathsf R}_{n} (z)\big)^{\prime} 
&=
T^{\mathsf R}_n (z) \big(M^{\mathsf R}_{n+1} (z)\big)^2
- \big( M^{\mathsf R}_{n} (z) \big)^2 T^{\mathsf R}_n (z) , 
\label{eq:zerocurvature4}
\end{align}
$n \in \mathbb N $, are satisfied.
%\end{enumerate}
\end{pro}

\begin{proof}
 Again we only give the proofs for the left case. 
We begin to prove that the sequence of matrix functions $\big\{ M^{\mathsf L}_n (z)\big\}_{n\in\N} $ is a sequence of matrices with coefficients given by entire functions.
 In fact,
 $\big( M^{\mathsf L}_n \big)_+ = 
\Big( \big(Z^{\mathsf L}_n\big)^{\prime}\Big)_+ 
 \Big( \big(Z^{\mathsf L}_n\big)^{-1}\Big)_+ $, and applying the \emph{constant} jump condition we get
\begin{gather*} 
\big( M^{\mathsf L}_n (z) \big)_+ = 
\Big( \big(Z^{\mathsf L}_n\big)^{\prime}\Big)_-
\begin{bmatrix} 
I & I \\ 
0_N & I 
\end{bmatrix}^{-1}
 \begin{bmatrix} I & I \\ 0_N & I \end{bmatrix} 
 \Big( \big(Z^{\mathsf L}_n\big)^{-1}\Big)_- = \big( M^{\mathsf L}_n (z) \big)_- .
\end{gather*}
It follows from the definition of $Z_n^{\mathsf L}$ that
\begin{gather*} 
 T^{\mathsf L}_n (z) = Y^{\mathsf L}_{n+1} (z) \big( Y^{\mathsf L}_{n} (z) \big)^{-1}= Z^{\mathsf L}_{n+1}(z)\big( Z^{\mathsf L}_{n} (z) \big)^{-1}.
\end{gather*}
Taking derivatives with respect to $z $ on $ T_n(z) $ we get
\begin{align*} 
\big( T_{n}^{\mathsf L} (z)\big)^{\prime}&= 
\big( Z_{n+1}^{\mathsf L}(z)\big)^{\prime} \big( Z_{n}^{\mathsf L}(z)\big)^{-1} 
- Z^{\mathsf L}_{n+1} (z) \big( Z_{n}^{\mathsf L}(z)\big)^{-1} 
\big( Z_{n}^{\mathsf L}(z)\big)^{\prime} \big( Z_{n}^{\mathsf L}(z)\big)^{-1}, & \ n &\in \N,
\end{align*}
and so, taking into account that
\begin{gather*} 
\big( Z_{n+1}^{\mathsf L}(z)\big)^{\prime} \big( Z_{n}^{\mathsf L}(z)\big)^{-1} 
= \big( Z_{n+1}^{\mathsf L}(z)\big)^{\prime} \big(Z_{n+1}^{\mathsf L}(z)\big)^{-1} Z_{n+1}^{\mathsf L}(z) \big( Z_{n}^{\mathsf L}(z)\big)^{-1} 
= T_{n+1}^\mathsf L M_n^\mathsf L \, ,
\end{gather*}
we get~\eqref{eq:zerocurvature1}. Using the same ideas we derive~\eqref{eq:zerocurvature2}.
Now, for~\eqref{eq:zerocurvature3} just replace the expressions for the derivative of the transfer matrix in~\eqref{eq:zerocurvature1}.
Multiplying~\eqref{eq:zerocurvature1} on the left by
$M_{n+1}^{\mathsf L}$ we~get
\begin{gather*}
M_{n+1}^{\mathsf L} 
\begin{bmatrix}
I & 0_N \\
0_N& 0_N
\end{bmatrix}%=\big( T_{n}^{\mathsf L}(z) \big)^{\prime} 
= \big(M^{\mathsf L}_{n+1} (z) \big)^2 T^{\mathsf L}_n(z) - 
\big(M_{n+1}^{\mathsf L} T^{\mathsf L}_n (z) \big) M^{\mathsf L}_{n} (z),
\end{gather*}
and again by~\eqref{eq:zerocurvature1} applied to the term $M_{n+1}^{\mathsf L} T^{\mathsf L}_n (z)$ we get~\eqref{eq:zerocurvature3}.
\end{proof}

Higher order transfer matrices 
\begin{align*}
T^\mathsf L_{n,\ell}(z)&
 : 
=T^\mathsf L_{n+\ell}(z)\cdots T^\mathsf L_{n}(z),&
T^\mathsf R_{n,\ell}(z)& 
 : 
=T^\mathsf R_{n}(z)\cdots T^\mathsf L_{n+\ell}(z),
\end{align*}
 satisfy
\begin{align*}
Y^\mathsf L_{n+\ell}(z)&=T^\mathsf L_{n,\ell}(z)Y^\mathsf L_{n}(z), &
Y^\mathsf R_{n+\ell}(z)&=Y^\mathsf R_{n}(z)T^\mathsf R_{n,\ell}(z).
\end{align*}
\begin{pro}
The following zero-curvature conditions hold, for all $n, \ell \in \mathbb N$,
\begin{align*}
\big( T_{n,\ell}^{\mathsf L}(z) \big)^{\prime} &= M^{\mathsf L}_{n+\ell+1} (z) T^{\mathsf L}_n(z) - T^{\mathsf L}_n (z) M^{\mathsf L}_{n} (z),&
\big(T^{\mathsf R}_{n,\ell} (z)\big)^{\prime} &=
T^{\mathsf R}_n (z) \, M^{\mathsf R}_{n+\ell+1} (z)
- M^{\mathsf R}_{n} (z) T^{\mathsf R}_n (z).
\end{align*}
\end{pro}
\begin{proof}
As before we only give a discussion for the left situation. It is done by induction, assuming that it holds for $\ell$ we prove it for $\ell+1$:
\begin{align*}
\big( T_{n,\ell+1}^{\mathsf L}(z) \big)^{\prime} &= \big( T_{n+\ell+1}^{\mathsf L}(z) T_{n,\ell}^{\mathsf L}(z) \big)^{\prime} 
=\big( T_{n+\ell+1}^{\mathsf L}(z) \big)^{\prime} T_{n,\ell}^{\mathsf L}(z)
+T_{n+\ell+1}^{\mathsf L}(z) \big( T_{n,\ell}^{\mathsf L}(z) \big)^{\prime} 
\\ &
=\begin{multlined}[t][0.75\textwidth]
\big( M^{\mathsf L}_{n+\ell+2} (z) T^{\mathsf L}_{n+\ell+1}(z) - T^{\mathsf L}_{n+\ell+1} (z) M^{\mathsf L}_{n+\ell+1} (z)\big) T_{n,\ell}^{\mathsf L}(z)
 \\+
T_{n+\ell+1}^{\mathsf L}(z) \big(M^{\mathsf L}_{n+\ell+1} (z) T^{\mathsf L}_{n,\ell}(z) - T^{\mathsf L}_{n,\ell} (z) M^{\mathsf L}_{n} (z)\big),
\end{multlined}
 \\ &
=
M^{\mathsf L}_{n+\ell+2} (z) T^{\mathsf L}_{n+\ell+1}(z) T_{n,\ell}^{\mathsf L}(z)
- T_{n+\ell+1}^{\mathsf L}(z) T^{\mathsf L}_{n,\ell} (z) M^{\mathsf L}_{n} (z),
\end{align*}
and the result is proven; recalling that for $\ell=0$ it is just the already proven zero-curvature condition.
\end{proof}

\begin{pro}[Computing the structure matrices]
If the subindex $\pmb +$ indicates that only the positive powers of the asymptotic expansion are kept, for each factorization $ W = W^{\mathsf L} W^{\mathsf R}$, we have for all $n \in \mathbb N$, the following power expansions for the structure matrices
\begin{align}\label{eq:MS}
%\begin{aligned}
M^{\mathsf L}_n (z) &= \Bigg(S^{\mathsf L}_n (z)
\begin{bmatrix} 
\big( W^{\mathsf L} (z) \big)^{\prime} \big( W^{\mathsf L} (z) \big)^{-1} & 0_N \\ 
0_N& -\big( W^{\mathsf R} (z) \big)^{-1} \big( W^{\mathsf R}(z) \big)^{\prime} 
\end{bmatrix} 
\big(S^{\mathsf L}_n(z)\big)^{-1} \Bigg)_{\pmb +},
 \\
\label{eq:MSR}
%\notag
{M}^{\mathsf R}_n (z)& = \Bigg(
\big({S}^{\mathsf R}_n (z)\big)^{-1} 
\begin{bmatrix}\big( W^{\mathsf R} (z) \big)^{-1} \big( W^{\mathsf R} (z) \big)^{\prime} & 0_N \\ 
0_N& - \big( W^{\mathsf L}(z) \big)^{\prime} \big( W^{\mathsf L} (z) \big)^{-1} 
\end{bmatrix} 
{S}^{\mathsf R}_n (z) \Bigg)_{\pmb +}
 .
%\end{aligned}
\end{align}
\end{pro}
\begin{proof}
Using assumption i) in Proposition~\ref{prop:Mn}, we find the expressions for the left structure matrix, $ M_n^{\mathsf L}(z)$, in terms of $ S_n^{\mathsf L} (z) $ and $ W(z)= W^{\mathsf L} (z) W^{\mathsf R}(z) $. For doing so we require %of 
the use of the definition of $ S_n^{\mathsf L}(z) $, i.e.
\begin{align*} %\label{eq:zz}
 Z^{\mathsf L}_n (z) = S^{\mathsf L}_n (z) 
 \begin{bmatrix} z^n W^{\mathsf L} (z) &0_N \\ 0_N & z^{-n} \big( W^{\mathsf R} (z) \big)^{-1} \end{bmatrix} , 
\end{align*}
and consequently, we find 
\begin{multline*} %\label{eq:expmn}
 M^{\mathsf L}_n (z) =
 \big( S_n^{\mathsf L} (z) \big)^{\prime} \big( S_n^{\mathsf L} (z) \big)^{-1} 
 \\ + 
S_n^{\mathsf L} (z) 
\begin{bmatrix}\big( W^{\mathsf L} (z) \big)^{\prime} \big( W^{\mathsf L} (z) \big)^{-1} +
%\dfrac
{n} {z}^{-1} & 0_N \\ 
0_N& -\big( W^{\mathsf R} (z) \big)^{-1} \big( W^{\mathsf R}(z) \big)^{\prime} -
%\dfrac
{n} {z}^{-1} 
\end{bmatrix} 
\big(S_n^{\mathsf L} (z)\big)^{-1}.
\end{multline*}
Given assumption iii) in the begining of this section, on the entire character of the right derivative, $\big( W^{\mathsf L} (z) \big)^{\prime} \big( W^{\mathsf L} (z) \big)^{-1}$, and of the left derivative, $\big( W^{\mathsf R} (z) \big)^{-1}\big( W^{\mathsf R}(z)\big)^{\prime}$, and since $ \big( S_n^{\mathsf L} (z) \big)^{\prime} \big( S_n^{\mathsf L} (z) \big)^{-1} $ 
have only negative powers of $z $ in its Laurent expansion, and given that the structure matrix $M^{\mathsf L}(z)$ has entire coefficients, the asymptotic expansion of $M_n^{\mathsf L} (z)$ about $\infty$ must be a power expansion. 

A similar approach holds for the right context, and we can determine $M_n^{\mathsf R}(z)$ in terms of $S^{\mathsf R}_n (z)$ and $ W(z)$. Indeed, from 
\begin{align*}
{Z}^{\mathsf R}_n (z)
&= \begin{bmatrix} W^{\mathsf R} (z) z^n& 0_N \\ 0_N & ( W^{\mathsf L} (z))^{-1} z^{-n}\end{bmatrix}
S^{\mathsf R}_n (z) , 
\end{align*}
we get
\begin{multline*}
%\label{eq:expmmn}
{M}^{\mathsf R}_n (z) = \big({S}^{\mathsf R}_n (z)\big)^{-1} \big(S_n^{\mathsf R}(z)\big)^{\prime} 
 \\ +
\big({S}^{\mathsf R}_n (z)\big)^{-1} 
\begin{bmatrix} \big( W^{\mathsf R} (z) \big)^{-1} \big( W^{\mathsf R} (z) \big)^{\prime} +
%\dfrac
{n}{z}^{-1} & 0_N \\ 
0_N& - \big( W^{\mathsf L}(z) \big)^{\prime} \big( W^{\mathsf L} (z) \big)^{-1} -
%\dfrac
{n}{z}^{-1} 
\end{bmatrix} 
{S}^{\mathsf R}_n (z),
\end{multline*}
and reasoning as for the left case we derive the desired result.
\end{proof}

Notice that given the matrices of entire functions $h^\mathsf L(z)$ and $h^\mathsf R(z)$ the structure matrices, using~\eqref{eq:MS}, can explicitly determined
in terms of the coefficients in $S_n^\mathsf L(z)$ and $S_n^\mathsf R(z)$. 
Moreover, when $h^\mathsf L(z),h^\mathsf R(z)\in\C^{N\times N}[z]$ are matrix polynomials, only the first elements, as much as the degree of the corresponding polynomial, in the asymptotic expansions of~$S_n^\mathsf L(z)$ and $S_n^\mathsf R(z)$ are involved, and we will have that $M^{\mathsf L}_n (z),M^{\mathsf R}_n (z) \in \C^{2N\times 2N}[z]$
are also polynomials with degree $\deg M^{\mathsf L}_n (z),\deg M^{\mathsf L}_n (z) =\max(h^{\mathsf L}_n (z) ,h^{\mathsf R}_n (z) )$.

\begin{rem}
 For the reductions we have 
 \begin{align*}
 M_n^\mathsf R(z)
 &=\big( M_n^\mathsf L(z) \big)^\top, & \text{symmetric},\\
 M_n^\mathsf R(z)
 &=\big( M_n^\mathsf L(\bar z) \big)^\dagger, & \text{Hermitian}.
 \end{align*}
 In both cases, we will use the notation
 \begin{align*}
 M_n(z):=M_n^\mathsf L(z).
 \end{align*}
\end{rem}

\section{Matrix Pearson equations and Differential equations}
\label{sec:4}

\subsection{Matrix Pearson equations }

As we have seen, the left and right logarithmic derivatives, $h^\mathsf L(z)=\big( W^{\mathsf L} (z) \big)^{\prime} \big( W^{\mathsf L} (z) \big)^{-1}$ and $h^\mathsf R(z)=\big( W^{\mathsf R} (z) \big)^{-1} \big( W^{\mathsf R}(z) \big)^{\prime} $, play an important role in the discussion of the structure matrices.
This motivates us to adopt the following strategy: 
assume that instead of a given matrix of weights we are provided with two matrices, say $h^{\mathsf L}(z)$ and $h^{\mathsf R}(z)$, of entire functions such that the following two matrix Pearson equations are satisfied
\begin{align}\label{eq:partial_Pearson_L}
\frac{\d W^\mathsf L}{\d z}&= h^{\mathsf L}(z) W^{\mathsf L}(z),
 \\
%\end{align}
%\begin{align}
\frac{\d W^\mathsf R}{\d z}&= W^{\mathsf R}(z) h^{\mathsf R}(z) ;\label{eq:partial_Pearson_R}
\end{align}
and given solutions to them we construct the corresponding matrix of weights $ W= W^\mathsf L W^\mathsf R$. Moreover, this matrix of weights is also characterized by a Pearson equation.
\begin{pro}[Pearson Sylvester differential equation] \label{prop:Pearson}
Given two matrices of entire functions $h^{\mathsf L}(z)$ and $h^{\mathsf R}(z)$, any solution of the Sylvester type matrix differential equation, which we call Pearson equation for the weight,
\begin{align}\label{eq:Pearson}
\frac{\d W}{\d z}=h^{\mathsf L} (z) W(z)+ W(z) h^{\mathsf R} (z)
\end{align} 
is of the form $ W= W^\mathsf L W^\mathsf R$ where the factor matrices $ W^\mathsf L$ and $ W^\mathsf R$ are solutions of~\eqref{eq:partial_Pearson_L} and~\eqref{eq:partial_Pearson_R}, respectively. 
\end{pro}

\begin{proof}
 Given solutions $ W^\mathsf L$ and $ W^\mathsf R$ of~\eqref{eq:partial_Pearson_L} and~\eqref{eq:partial_Pearson_R}, respectively, it follows intermediately, just using the Leibniz law for derivatives, that $ W= W^\mathsf L W^\mathsf R$ fulfills~\eqref{eq:Pearson}. Moreover, given a solution $ W$ of~\eqref{eq:Pearson} we pick a solution $ W^\mathsf L$ of~\eqref{eq:partial_Pearson_L}, then it is easy to see that $( W^\mathsf L)^{-1} W$ satisfies~\eqref{eq:partial_Pearson_R}.
\end{proof}

\begin{rem}
The matrix of weights $ W$ does not uniquely determine the left and right factors; indeed if $ W= W^\mathsf L W^\mathsf R$, with factors solving~\eqref{eq:partial_Pearson_L} and~\eqref{eq:partial_Pearson_R}, respectively, then
$\widetilde W^\mathsf L = W^\mathsf L C$ and $\widetilde W^\mathsf R=C^{-1} W^\mathsf R$ for~$C$ a nonsingular matrix, gives also another possible factorization $ W=\widetilde W^\mathsf L\widetilde W^\mathsf R$, with factors solving the partial Pearson equations~\eqref{eq:partial_Pearson_L} and~\eqref{eq:partial_Pearson_R}. This indeterminacy disappears when one considers the right and left derivatives of the factors. 
\end{rem}

\begin{rem}
 Given two matrices of entire functions $h^{\mathsf L}(z)$ and $h^{\mathsf R}(z)$ and a matrix of weights~$ W$ characterized by the matrix Pearson equation~\eqref{eq:Pearson} we have the left and right fundamental matrices $Y_n^\mathsf L(z)$ and $Y^\mathsf R_n(z)$ satisfying corresponding Riemann--Hilbert problems. The %corresponding 
associated structure matrices are from~\eqref{eq:MS} and~\eqref{eq:MSR} given by,
\begin{align} \label{eq:MSP}
%\begin{aligned}
M^{\mathsf L}_n (z) &= \Bigg(S^{\mathsf L}_n (z)
\begin{bmatrix} 
h^\mathsf L(z)& 0_N \\ 
0_N& -h^\mathsf R(z)
\end{bmatrix} 
\big(S^{\mathsf L}_n(z)\big)^{-1} \Bigg)_{\pmb +}, 
 \\ %&
\label{eq:MSPR}
{M}^{\mathsf R}_n (z)& = \Bigg(
\big({S}^{\mathsf R}_n (z)\big)^{-1} 
\begin{bmatrix}h^\mathsf R(z) & 0_N \\ 
0_N& - h^\mathsf L(z)
\end{bmatrix} 
{S}^{\mathsf R}_n (z) \Bigg)_{\pmb +}.
%\end{aligned}
\end{align}
\end{rem}

\begin{rem}
 For the symmetric and Hermitian reductions, we have
 \begin{align*}
 h^\mathsf R(z)&=\big( h^\mathsf L(z)\big)^\top,&\text{symmetric}, \\
h^\mathsf R(z)&=\big( h^\mathsf L(\bar z)\big)^\dagger,&\text{Hermitian},
 \end{align*}
 and 
 \eqref{eq:partial_Pearson_L} and \eqref{eq:partial_Pearson_R} collapses into a single equation 
  \begin{align*}
 \frac{\d \rho}{\d z}&= h(z) \rho(z),
 \end{align*}
where 
$h(z):=h^\mathsf L(z)$,
 and the Pearson equation \eqref{eq:Pearson} reads
 \begin{align}\label{eq:Pearson:_reduced}
\begin{aligned}
 \frac{\d W}{\d z}&=h(z) W(z)+ W(z) (h(z))^\top,& \quad \text{symmetric}, \\
 \frac{\d W}{\d z}&=h(z) W(z)+ W(z) (h(\bar z))^\dagger,& \quad  \text{Hermitian}.
\end{aligned}
 \end{align} 
\end{rem}

\subsection{Sylvester differential equations for the fundamental matrices}

Following the standard use in Soliton Theory, given a matrix of holomorphic functions $A(z)$ we define its Miura transform by 
\begin{align*}
 \mathcal M(A)=A^{\prime}(z)+(A(z))^2.
\end{align*}
Observe that when $A$ is a right (left) logarithmic
derivative $A=w^{\prime}w^{-1}$ ($A=w^{-1}w^{\prime}$) we have $\mathcal M(A)=w^{\prime\prime} w^{-1}$ ($\mathcal M(A)=w^{-1} w^{\prime\prime}$).
 \begin{pro}[Sylvester differential linear systems] 
In the conditions of Proposition~\ref{prop:Pearson}, the left fundamental matrix $Y^\mathsf L_n (z)$ and the right fundamental matrix $Y^\mathsf R_n (z)$ satisfy, for each $n \in \mathbb N$, the following Sylvester matrix differential equations,
\begin{align} \label{eq:estrutura1}
\big(Y^\mathsf L_n (z)\big)^{\prime} 
= M^\mathsf L_n (z) Y^\mathsf L_n (z) - Y^\mathsf L_n (z) 
\begin{bmatrix} 
h^{\mathsf L}(z) & 0_N\\ 
0_N & -h^{\mathsf R} (z) 
\end{bmatrix},\\
 \label{eq:estrutura2}
\big(Y^\mathsf R_n (z)\big)^{\prime} 
= Y^\mathsf R_n (z) M^\mathsf R_n (z) - \begin{bmatrix} 
h^{\mathsf R}(z) & 0_N\\ 
0_N & -h^{\mathsf L} (z) 
\end{bmatrix}Y^\mathsf R_n (z) ,
\end{align}
respectively.
\end{pro}
\begin{proof}
As
$M^\mathsf L_n (z) =
\big(Z^\mathsf L_n(z)\big)^{\prime} \big(Z^\mathsf L_n(z)\big)^{-1}$ 
is the right derivative of the constant jump structure matrix from~\eqref{eq:zn}
we get~\eqref{eq:estrutura1};~\eqref{eq:estrutura2} is proven analogously.
\end{proof}

We write
\begin{align*}
 M^\mathsf L_n (z) &= 
\begin{bmatrix} M_{1,1,n}^\mathsf L (z) & M_{1,2,n}^\mathsf L(z) \\[.05cm]
M_{2,1,n}^\mathsf L (z) & M_{2,2,n}^\mathsf L (z) \end{bmatrix}, 
 &
M^\mathsf R_n (z) &= 
\begin{bmatrix} M_{1,1,n}^\mathsf R (z) & M_{1,2,n}^\mathsf R(z) \\[.05cm]
M_{2,1,n}^\mathsf R (z) & M_{2,2,n}^\mathsf R (z) 
\end{bmatrix}, 
\end{align*}
to express the previous results in the following manner.
\begin{coro}%{pro}
The Sylvester matrix differential equations~\eqref{eq:estrutura1} and~\eqref{eq:estrutura2} split in the following Sylvester differential systems
\begin{align}
\label{eq:structure_PL}
%\hspace{-.4cm}
\begin{cases}
%\begin{aligned}
\big(P^\mathsf L_n (z) \big)^{\prime}+ P^\mathsf L_n (z) h^\mathsf L(z) %&
= M_{1,1,n}^\mathsf L (z) P^\mathsf L_n (z) - M_{1,2,n}^\mathsf L (z) C_{n-1} P^\mathsf L_{n-1} (z) , \\
\big(P^\mathsf L_{n-1} (z) \big)^{\prime} + P^\mathsf L_{n-1} (z) h^\mathsf L(z) %&
= - C_{n-1}^{-1} M_{2,1,n}^\mathsf L (z) P^\mathsf L_n (z) + C_{n-1}^{-1} M_{2,2,n}^\mathsf L (z) C_{n-1} P^\mathsf L_{n-1} (z) , 
%\end{aligned}
\end{cases} 
 \\
%\end{align}
%\begin{align}
%\label{eq:structure_QL}
\label{eq:structure_PR}
%\begin{aligned}
\begin{cases}
%\begin{aligned}
\big(Q^\mathsf L_n (z) \big)^{\prime}+Q^\mathsf L_n (z) h^\mathsf R(z) %&
 = M_{1,1,n}^\mathsf L Q^\mathsf L_n (z) - M_{1,2,n}^\mathsf L(z) C_{n-1} Q^\mathsf L_{n-1} (z) , \\
\big(Q^\mathsf L_{n-1} (z) \big)^{\prime}+Q^\mathsf L_{n-1} (z) h^\mathsf R(z) %&
 = - C_{n-1}^{-1} M_{2,1,n}^\mathsf L (z) Q^\mathsf L_n (z) + C_{n-1}^{-1} M_{2,2,n}^\mathsf L (z) C_{n-1} Q^\mathsf L_{n-1} (z) , 
%\end{aligned}
\end{cases} 
 \\
%\end{align}
%\begin{align}
\notag
\begin{cases}
%\begin{aligned}
\big(P^\mathsf R_n (z) \big)^{\prime}+ h^\mathsf R(z)P^\mathsf R_n (z) %&
 = P^\mathsf R_n (z) M_{1,1,n}^\mathsf R (z) - P^\mathsf R_{n-1} (z) C_{n-1} M_{2,1,n}^\mathsf R(z) , \\
\big(P^\mathsf R_{n-1} (z) \big)^{\prime} + h^\mathsf R(z)P^\mathsf R_{n-1} (z) %&
 = -P^\mathsf R_n (z) M_{1,2,n}^\mathsf R (z) C_{n-1}^{-1} + P^\mathsf R_{n-1} (z) C_{n-1} M_{2,2,n}^\mathsf R (z) C_{n-1}^{-1} , 
%\end{aligned}
\end{cases} 
 \\
%\end{align}
%\begin{align}
%\label{eq:structure_QR}
\notag
\begin{cases}
%\begin{aligned}
\big(Q^\mathsf R_n (z) \big)^{\prime} + h^\mathsf L(z)Q^\mathsf R_n (z) %&
 = Q^\mathsf R_n (z) M_{1,1,n}^\mathsf R (z) - Q^\mathsf R_{n-1} (z) C_{n-1} M_{2,1,n}^\mathsf R(z) , \\
\big(Q^\mathsf R_{n-1} (z) \big)^{\prime} + h^\mathsf L(z)Q^\mathsf R_{n-1} (z) %&
 = -Q^\mathsf R_n (z) M_{1,2,n}^\mathsf R (z) C_{n-1}^{-1} + Q^\mathsf R_{n-1} (z) C_{n-1} M_{2,2,n}^\mathsf R (z) C_{n-1}^{-1} , 
%\end{aligned}
\end{cases}
%\end{aligned}
\end{align}
\end{coro} %{pro}

We first observe from the linear differential systems~\eqref{eq:structure_PL} and~\eqref{eq:structure_PR} satisfied by the left and right matrix orthogonal polynomials, respectively, we will be able to extract in some scenarios, see next section on applications, a matrix eigenvalue problem for a second order matrix differential operator, with matrix eigenvalues. The differential systems~\eqref{eq:structure_PL} and~\eqref{eq:structure_PR} for the left and right second kind functions also provide interesting information, and we will use them discover nonlinear equations satisfied by the recursion coefficients.

\begin{rem}
 For the reductions we have
 \begin{align*} 
 \big(Y_n (z)\big)^{\prime} 
& = M_n (z) Y_n (z) - Y_n (z) 
 \begin{bmatrix} 
 h(z) & 0_N\\ 
 0_N & -(h (z) )^\top
 \end{bmatrix}, &\text{symmetric},\\
 \big(Y_n (z)\big)^{\prime} 
 &= M_n (z) Y_n (z) - Y_n (z) 
 \begin{bmatrix} 
 h(z) & 0_N\\ 
 0_N & -(h (\bar z) )^\dagger.
 \end{bmatrix}, & \text{Hermitian}.
 \end{align*}
\end{rem}

\section{Second order differential operators} \label{sec:5}

We firstly derive, as a consequence of the Sylvester differential linear systems, second order differential equations fulfilled by the fundamental matrices, and therefore by the matrix biorthogonal polynomials and also by the corresponding second kind functions.

\begin{pro}[Second order linear differential equations]\label{teo:second_order_diff}
In the conditions of Proposition~\ref{prop:Pearson}, the sequence of fundamental matrices, $\big\{ Y_n^{\mathsf L} \big\}_{n \in \mathbb N}$ and $\big\{ Y_n^{\mathsf R} \big\}_{n \in \mathbb N}$, satisfy 
 \begin{align}\label{eq:edo1}
\begin{multlined}[t][0.9\textwidth]
 \big(Y^\mathsf L_n (z)\big)^{\prime\prime}+2 \big(Y^\mathsf L_n (z)\big)^{\prime}\begin{bmatrix} 
 h^{\mathsf L}(z) & 0_N \\ 
 0_N & -h^{\mathsf R} (z) 
 \end{bmatrix}+ Y_n^\mathsf L (z)
 \begin{bmatrix} 
 \mathcal M\big(h^{\mathsf L}(z)\big)& 0_N\\ 
 0_N &\mathcal M\big(-h^{\mathsf R}(z)\big)
 \end{bmatrix}
 %\\= %\Big(\big(M^\mathsf L_n (z) \big)^{\prime}+\big(M^\mathsf L_n (z) \big)^2\Big) 
 \\ 
=\mathcal M\big(M^\mathsf L_n (z) \big)Y^\mathsf L_n (z) ,
\end{multlined}
 \\
%\end{align}
%\begin{align}
 \label{eq:edo2}
\begin{multlined}[t][0.9\textwidth]
 \big(Y^\mathsf R_n (z)\big)^{\prime\prime}+2\begin{bmatrix} 
 h^{\mathsf R}(z) & 0_N\\ 
 0_N & -h^{\mathsf L} (z) 
 \end{bmatrix} \big(Y^\mathsf R_n (z)\big)^{\prime}+ 
 %\begin{bmatrix} 
 %\big(h^{\mathsf R}(z)\big)^{\prime} +\big(h^{\mathsf R}(z)\big)^2& 0_N\\ 
 %0_N & -\big(h^{\mathsf L}(z)\big)^{\prime} +\big(h^{\mathsf L}(z)\big)^2 
 %\end{bmatrix}
 \begin{bmatrix} 
 \mathcal M\big(h^{\mathsf R}(z)\big)& 0_N\\ 
 0_N &\mathcal M\big(-h^{\mathsf L}(z)\big)
 \end{bmatrix}
 Y_n^\mathsf L (z)
 \\ = 
Y^\mathsf R_n (z) \mathcal M\big(M^\mathsf R_n (z) \big).
 %\Big(\big(M^\mathsf R_n (z) \big)^{\prime}+\big(M^\mathsf R_n (z) \big)^2\Big) ,
\end{multlined}
 \end{align}
\end{pro}
\begin{proof}
We prove~\eqref{eq:edo1}. First, let us take a derivative of~\eqref{eq:estrutura1} to get 
\begin{multline*}
\big(Y^\mathsf L_n (z)\big)^{\prime\prime}+\big(Y^\mathsf L_n (z) \big)^{\prime} 
 \begin{bmatrix} 
 h^{\mathsf L}(z) & 0_N\\ 
 0_N & -h^{\mathsf R} (z) 
 \end{bmatrix}+Y^\mathsf L_n (z) 
 \begin{bmatrix} 
 \big(h^{\mathsf L}(z)\big)^{\prime} & 0_N\\ 
 0_N & -\big(h^{\mathsf R} (z) \big)^{\prime} 
 \end{bmatrix}
 \\ =
\big(M^\mathsf L_n (z)\big)^{\prime} Y^\mathsf L_n (z) +M^\mathsf L_n (z) \big(Y^\mathsf L_n (z)\big)^{\prime} 
\end{multline*}
but again by~\eqref{eq:estrutura1}
\begin{align*}
 M^\mathsf L_n (z) \big(Y^\mathsf L_n (z)\big)^{\prime} 
%&=M^\mathsf L_n (z) \bigg(M^\mathsf L_n (z) Y^\mathsf L_n (z) - Y^\mathsf L_n (z) 
%\begin{bmatrix} 
%h^{\mathsf L}(z) & 0_N\\ 
%0_N & -h^{\mathsf R} (z) 
%\end{bmatrix}\bigg)\\
&=\big(M^\mathsf L_n (z) \big)^2 Y^\mathsf L_n (z) - M^\mathsf L_n (z)Y^\mathsf L_n (z) 
\begin{bmatrix} 
 h^{\mathsf L}(z) & 0_N \\ 
 0_N & -h^{\mathsf R} (z) 
\end{bmatrix} %\\
\end{align*}
and if we substitute
\begin{align*}
M^\mathsf L_n (z)Y^\mathsf L_n (z)
&=\big(Y^\mathsf L_n (z)\big)^{\prime} 
 + Y^\mathsf L_n (z) 
 \begin{bmatrix} 
 h^{\mathsf L}(z) & 0_N \\ 
 0_N & -h^{\mathsf R} (z) 
 \end{bmatrix} %\\
\end{align*}
we finally get
\begin{align*}
M^\mathsf L_n (z) \big(Y^\mathsf L_n (z)\big)^{\prime} 
&=\big(M^\mathsf L_n (z) \big)^2 Y^\mathsf L_n (z) - \big(Y^\mathsf L_n (z)\big)^{\prime} \begin{bmatrix} 
h^{\mathsf L}(z) & 0_N \\ 
0_N & -h^{\mathsf R} (z) 
\end{bmatrix}
- Y^\mathsf L_n (z) 
\begin{bmatrix} 
 h^{\mathsf L}(z) & 0_N\\ 
 0_N & -h^{\mathsf R} (z) 
\end{bmatrix}^2,
 \end{align*}
and the result follows.
\end{proof}% 
\begin{defi}
 For the next corollary we need to introduce the following $\C^{2N\times 2N}$ valued functions in terms of the difference of two Miura maps
\begin{align}
\label{eq:Miura_bispectrality_L}
 \mathsf H_n^{\mathsf L} (z) =
 \begin{bmatrix}
 \mathsf H_{1,1,n}^{\mathsf L} (z) & \mathsf H_{1,2,n}^{\mathsf L} (z)\\[.05cm] 
 \mathsf H_{2,1,n}^{\mathsf L} (z) & \mathsf H_{2,2,n}^{\mathsf L} (z)
 \end{bmatrix}
 &
 :
=\mathcal M(M_n^{\mathsf L} (z)) 
- \mathcal M 
\left( \begin{bmatrix}
h^{\mathsf L} (z) & 0_N \\ 0_N & -h^{\mathsf R} (z)
\end{bmatrix} \right) ,
 \\
%\end{align}
%\begin{align}
\label{eq:Miura_bispectrality_R}
\mathsf H_n^{\mathsf R} (z) =\begin{bmatrix}
\mathsf H_{1,1,n}^{\mathsf R} (z) & \mathsf H_{1,2,n}^{\mathsf R} (z)\\[.05cm]
 \mathsf H_{2,1,n}^{\mathsf R} (z) & \mathsf H_{2,2,n}^{\mathsf R} (z)
\end{bmatrix}&=\mathcal M(M_n^{\mathsf R} (z)) 
-\mathcal M 
\left( \begin{bmatrix}
h^{\mathsf R} (z) & 0_N \\ 0_N & -h^{\mathsf L} (z)
\end{bmatrix} \right).
\end{align}
\end{defi}

\begin{coro}
The second order matrix differential equations~\eqref{eq:edo1} and~\eqref{eq:edo2} split in the following differential relations
\begin{align}
\label{eq:secondorderpl}
\begin{multlined}[t][0.85\textwidth]
\big(P_n^{\mathsf L} \big)^{\prime \prime}(z) + 2 \big( P_n^{\mathsf L} \big)^{ \prime}(z) h^{\mathsf L} (z)
+
P_n^{\mathsf L} (z) \mathcal M(h^{\mathsf L} (z))
 \\ = \big( {\mathcal M }(h^{\mathsf L} (z) ) + \mathsf H_{1,1,n}^{\mathsf L} (z) \big) P_n^{\mathsf L} (z) - \mathsf H_{1,2,n}^{\mathsf L} (z) P_{n-1}^{\mathsf L} (z) \, , 
\end{multlined}
 \\
\label{eq:secondorderql}
\begin{multlined}[t][0.85\textwidth]
\big(Q_n^{\mathsf L} \big)^{\prime \prime}(z) - 2\big(Q_n^{\mathsf L} \big)^{ \prime}(z) h^{\mathsf R} (z) 
+
Q_n^{\mathsf L} (z)\mathcal M(-h^{\mathsf R} (z)) 
 \\ = 
\big( {\mathcal M }(h^{\mathsf L} (z)) + \mathsf H_{1,1,n}^{\mathsf L} (z) \big) Q_n^{\mathsf L} (z) - \mathsf H_{1,2,n}^{\mathsf L} (z) Q_{n-1}^{\mathsf L} (z) \, ,
\end{multlined}
 \\
\label{eq:secondorderpr}
\begin{multlined}[t][0.85\textwidth]
 \big(P_n^{\mathsf R} \big)^{\prime \prime}(z) + 2 h^{\mathsf R} (z)\big(P_n^{\mathsf R} (z) \big)^{ \prime}(z) 
+
\mathcal M(h^{\mathsf R} (z)) P_n^{\mathsf R} (z)
 \\ = 
P_n^{\mathsf R} (z) \big( {\mathcal M }(h^{\mathsf R} (z) ) + \mathsf H_{1,1,n}^{\mathsf R} (z) \big)
- P_{n-1}^{\mathsf R} (z) \mathsf H_{2,1,n}^{\mathsf R} (z) \, ,
\end{multlined}
 \\
\label{eq:secondorderqr}
\begin{multlined}[t][0.85\textwidth]
 \big(Q_n^{\mathsf R} \big)^{\prime \prime}(z) - 2 h^{\mathsf L} (z)\big(Q_n^{\mathsf R} \big)^{ \prime}(z) 
+
\mathcal M(-h^{\mathsf L} (z)) Q_n^{\mathsf R} (z)
 \\ = 
Q_n^{\mathsf R} (z) \big( {\mathcal M }(h^{\mathsf R} (z) ) + \mathsf H_{1,1,n}^{\mathsf R} (z) \big)
- Q_{n-1}^{\mathsf R} (z) \mathsf H_{2,1,n}^{\mathsf R} (z) .
\end{multlined}
\end{align}
\end{coro}

\begin{proof}
Is a direct consequence of Proposition~\ref{teo:second_order_diff}.
\end{proof}

\subsection{Adjoint operators} \label{subsec:nova}

We now elaborate around the idea of adjoint operators in this matrix scenario. 
\begin{defi}
 Given linear operator $L\in L(\C^{N\times N}[z])$ and a matrix of weights $ W(z)$, its adjoint operator $L^*$ is an operator such that
\begin{align*}
\langle L(P), \tilde P \rangle_{ W}&=\langle P, L^*(\tilde P )\rangle_{ W}, & 
 %\forall 
P(z),\tilde P(z)\in \C^{N\times N}[z],
\end{align*}
in terms of the sesquiliner form introduced in~\eqref{eq:sesquilinear}.
\end{defi}

Care must be taken at this point because in this definition of adjoint of a matrix differential operator we are not taken the transpose or the Hermitian conjugate of the matrix coefficients as was done in~\cite{duran_3}.

\begin{defi}
Motivated by~\eqref{eq:secondorderpl} and \eqref{eq:secondorderpr} we introduce two linear operators $\pmb \ell^{\mathsf L}$ and $\pmb \ell^{\mathsf R}$, acting on the linear space of polynomials $\C^{N \times N} [z]$ as follows
\begin{align*}
\pmb \ell^{\mathsf L} (P)& :
= P^{\prime\prime} + 2 P^{\prime} h^{\mathsf L} + P \mathcal M (h^{\mathsf L}), &
\pmb \ell^{\mathsf R} (P) & :
= P^{\prime\prime} + 2 h^{\mathsf R} P^{\prime} + \mathcal M (h^{\mathsf R}) P .
\end{align*}
\end{defi}

\begin{lemma} \label{prop:eqrho}
Let us assume that the matrix of weights $ W (z) %= W^\mathsf L (z) W^{\mathsf R} (z)
$ do satisfy the following boundary conditions
\begin{align}\label{eq:bconditions}
 W|_{\partial\gamma}&=0_N, &\big( W^{\prime} - 2h^{\mathsf L} W\big)\big|_{\partial \gamma}&=0_N, &
\big( W^{\prime} - 2 W h^{\mathsf R} \big)\big|_{\partial\gamma}&=0_N,
\end{align}
where $\partial\gamma$ is the boundary of the curve $\gamma$, i.e. its endpoints.
%\begin{gather} \label{eq:bconditions}
% W \, , W^{\prime} - 2h^{\mathsf L} W \ and \ \ W^{\prime} - 2 W h^{\mathsf R} 
%\ \mbox{ are zero at the end points of } \ \ \gamma.
%\end{gather}
Then,
$ W(z)$ satisfies a Pearson Sylvester differential equation~\eqref{eq:Pearson} if, and only if,
$ W(z)$ satisfies the following second order matrix differential equations
\begin{gather}
\label{eq:rhoml}
{ W}^{\prime \prime} -2 \big( h^{\mathsf L} W\big)^{ \prime} + \mathcal M ( h^{\mathsf L} ) W = W\mathcal M ( h^{\mathsf R} ) \, , \\
\label{eq:rhomr}
{ W}^{\prime \prime} -2 \big( W h^{\mathsf R} \big)^{ \prime} + W\mathcal M ( h^{\mathsf R} ) = \mathcal M ( h^{\mathsf L} ) W \, .
\end{gather}
\end{lemma}

\begin{proof}
Taking derivative on~\eqref{eq:Pearson}, we get 
\begin{gather*}
{ W}^{\prime \prime} =\mathcal M ( h^{\mathsf L} ) W + W \mathcal M ( h^{\mathsf R} ) + 2 h^{\mathsf L} W h^{\mathsf R}
\end{gather*}
But, it is easy to see that
\begin{align*}
\big( h^{\mathsf L} W\big)^{ \prime} &=\mathcal M ( h^{\mathsf L} ) W + h^{\mathsf L} W h^{\mathsf R} , &
\big( W h^{\mathsf R} \big)^{ \prime} &= W\mathcal M ( h^{\mathsf R} ) + h^{\mathsf L} W h^{\mathsf R},
\end{align*} 
and so we arrive to~\eqref{eq:rhoml} and~\eqref{eq:rhomr}.

The reciprocally result is a consequence of adding the equations~\eqref{eq:rhoml},~\eqref{eq:rhomr} and the boundary conditions~\eqref{eq:bconditions}.
\end{proof}

Now, we will see that these two operators are adjoint to each other with respect to the sesquilinear form induced by the weight functions $ W$.
\begin{pro} \label{prop:aa}
Whenever $ W (z)$ satisfies~\eqref{eq:Pearson} and the boundary conditions~\eqref{eq:bconditions}, we have that
\begin{align}\label{eq:adjoint_ell}
 \pmb \ell^{\mathsf R}= \big(\pmb \ell^{\mathsf L}\big)^*,
\end{align} 
 %i.e. 
or, equivalently,
\begin{align*}
\langle \pmb \ell^{\mathsf L} ( P) , \tilde P\rangle_ W& = \langle P , \pmb \ell^{\mathsf R} ( \tilde P ) \rangle_ W, &P(z),\tilde P (z)&\in \mathbb \C^{N\times N}[z] .
\end{align*}
\end{pro}

\begin{proof}
By using the linearity of these operators it is sufficient to prove
\begin{gather*}
\langle \pmb \ell^{\mathsf L} ( P_n^{\mathsf L} ) \, , \, P_k^{\mathsf R} \rangle_ W = \langle P_n^{\mathsf L} \, , \, \pmb \ell^{\mathsf R} ( P_k^{\mathsf R} ) \rangle_ W
\, , \ \ n , k \in \mathbb N \, .
\end{gather*}
If we omit, for the sake of simplicity, the $z$ dependece of the integrands in the integrals, we have
\begin{align*}
\langle \pmb \ell^{\mathsf L} ( P_n^{\mathsf L} ) \, , \, P_k^{\mathsf R} \rangle_ W 
 = 
\int_{\gamma} (P_n^{\mathsf L})^{\prime\prime} \, W \, P_k^{\mathsf R} \, \operatorname d z
+ 2\int_{\gamma} (P_n^{\mathsf L})^{\prime} \, ( h^{\mathsf L} \, W ) \, P_k^{\mathsf R} \, \operatorname d z
+ \int_{\gamma} P_n^{\mathsf L} \, 
\mathcal M ( h^{\mathsf L} ) \, W \, P_k^{\mathsf R} \, \operatorname d z \, ,
\end{align*}
and, using integration by parts, we find
\begin{align*}
\langle \pmb \ell^{\mathsf L} ( P_n^{\mathsf L} ) , P_k^{\mathsf R} \rangle_ W
 & = \big((P_n^{\mathsf L})^{\prime} W P_k^{\mathsf R}\big) \big|_{\partial\gamma} 
 - \int_{\gamma} (P_n^{\mathsf L})^{\prime} \Big( \big( W P_k^{\mathsf R}\big)^{\prime} - 2h^{\mathsf L} W \Big) P_k^{\mathsf R} \d z
+ \int_{\gamma} P_n^{\mathsf L} 
\mathcal M ( h^{\mathsf L} ) W P_k^{\mathsf R} \d z
 \\
 & \hspace{-.085cm} \begin{multlined}[t][0.8\textwidth]
 =
 \big((P_n^{\mathsf L})^{\prime} W P_k^{\mathsf R}\big) \big|_{\partial\gamma} - \Big(P_n^{\mathsf L} \Big( \big( W P_k^{\mathsf R}\big)^{\prime} - 2h^{\mathsf L} W \Big) P_k^{\mathsf R}\Big) \Big|_{\partial\gamma} \\
+ \int_{\gamma} P_n^{\mathsf L} \, \big( ( W \, P_k^{\mathsf R} )^{\prime\prime}
- 2 \, ( h^{\mathsf L} \, W \, P_k^{\mathsf R} )^{\prime}
+ \mathcal M ( h^{\mathsf L} ) \, W \, P_k^{\mathsf R} \big) 
\, \operatorname d z \, .
 \end{multlined}
\end{align*}
Now, considering the boundary conditions~\eqref{eq:bconditions}
%\begin{gather*}
%\left. W \right|_{\gamma} =0 \, , \ \ \
%\left. \big( W^{\prime} - 2h^{\mathsf L} \, W ) \,\right|_{\gamma} =0
%\end{gather*}
and taking into account that
\begin{gather*}
( W \, P_k^{\mathsf R} )^{\prime\prime} = W^{\prime\prime} \, P_k^{\mathsf R} 
+ 2 \, W^{\prime} \, (P_k^{\mathsf R})^{\prime} + W \, (P_k^{\mathsf R})^{\prime\prime} \, , \ \ \
( h^{\mathsf L} \, W \, P_k^{\mathsf R} )^{\prime} = (h^{\mathsf L} \, W)^{\prime} \, P_k^{\mathsf R}
+ ( h^{\mathsf L} \, W) \, (P_k^{\mathsf R})^{\prime} \, ,
 \end{gather*} 
we arrive to
\begin{multline*}
\langle \pmb \ell^{\mathsf L} ( P_n^{\mathsf L} ) \, , \, P_k^{\mathsf R} \rangle_ W 
 =
\int_{\gamma} P_n^{\mathsf L} 
\big( W^{\prime\prime} 
- 2 (h^{\mathsf L} \, W)^{\prime}
+ \mathcal M ( h^{\mathsf L} ) W 
\big) P_k^{\mathsf R} \d z
 \\
+
2 \int_{\gamma} P_n^{\mathsf L} 
\big( W^{\prime} 
- h^{\mathsf L} W
\big) ( P_k^{\mathsf R} )^{\prime} \d z
+ \int_{\gamma} P_n^{\mathsf L} W ( P_k^{\mathsf R} )^{\prime\prime} \d z;
\end{multline*}
and so
\begin{align*}
\langle \pmb \ell^{\mathsf L} ( P_n^{\mathsf L} ) \, , \, P_k^{\mathsf R} \rangle_ W 
 &=
\int_{\gamma} P_n^{\mathsf L} W 
\big( (P_k^{\mathsf R} )^{\prime\prime}
+ 2 \, h^{\mathsf R} (P_k^{\mathsf R})^{\prime}
+ \mathcal M (h^{\mathsf R}) P_k^{\mathsf R}
\big) \d z,& 
%\forall 
n,k&\in\{0,1,2,\dots\} 
\end{align*}
or, equivalently,
\begin{gather*}
\langle \pmb \ell^{\mathsf L} ( P_n^{\mathsf L} ) \, , \, P_k^{\mathsf R} \rangle_ W 
 =
\langle P_n^{\mathsf L} \, , \, \pmb \ell^{\mathsf R} ( P_k^{\mathsf R} ) \rangle_ W
%+ \int_{\gamma} P_n^{\mathsf L} \, W \, P_m^{\mathsf R} \, \mathcal M ( h^{\mathsf R} ) \, \operatorname d z
%- \int_{\gamma} \mathcal M ( h^{\mathsf L} ) \, P_n^{\mathsf L} \, W \, P_m^{\mathsf R} \, \operatorname d z 
 \, ,
\end{gather*}
which completes the proof.
\end{proof}
\begin{rem}
 For the symmetric or Hermitian reductions we find that
 \begin{align*}
 \pmb \ell^{\mathsf R} (P)&= \big(\pmb \ell^{\mathsf L}(P^\top)\big)^\top, & \text{symmetric},\\
  \pmb \ell^{\mathsf R} (P)&= \big(\pmb \ell^{\mathsf L}(P^\dagger)\big)^\dagger, & \text{Hermitian},
 \end{align*}
 where in the last case we take $x\in\R$. Relation \eqref{eq:adjoint_ell} reads in this case as follows
 \begin{align*}
 \pmb\ell^* (P)&=(\pmb\ell(P^\top))^\top, &\text {symmetric},\\
  \pmb\ell^* (P)&=(\pmb\ell(P^\dagger))^\dagger, &\text {Hermitian};
 \end{align*}
 for $P$ any matrix polynomial and $\pmb\ell:=\pmb \ell^{\mathsf L}$.
\end{rem}

\begin{defi}
 Let $\alpha^\mathsf L$ and $\alpha^\mathsf R$ be two $N \times N$ matrices and define the following linear operators acting on the space of matrix polynomials $ \mathbb \C^{N\times N}[z]$ as follows
 \begin{align*} %\label{eq:operadorl}
 {\mathcal L}^{\mathsf L} (P) &:= P^{\prime\prime} + 2 P^{\prime} h^{\mathsf L} + P {\alpha}^{\mathsf L}, &
 %\label{eq:operadorr}
 {\mathcal L}^{\mathsf R} (P) &:=P^{\prime\prime} + 2 h^{\mathsf R} P^{\prime} + {\alpha}^{\mathsf R} P.
 \end{align*}
\end{defi}

 Observe that
\begin{align*} %\label{eq:operadorl}
{\mathcal L}^{\mathsf L} (P) &= \pmb \ell^{\mathsf L} (P)- P \, {\mathcal M} (h^{\mathsf L}) + P {\alpha}^{\mathsf L},& 
%\label{eq:operadorr}
{\mathcal L}^{\mathsf R} (P) &=\pmb \ell^{\mathsf R} (P - \mathcal M (h^{\mathsf R}) P + {\alpha}^{\mathsf R} P.
\end{align*}

We have the following characterization.

\begin{teo} \label{teo:novo}
The following conditions are equivalent:
%\begin{enumerate}
%\item 

\hangindent=.9cm \hangafter=1
{\noindent}$\phantom{ol}${\rm i)}
${\mathcal L}^{\mathsf R}=\big({\mathcal L}^{\mathsf L} \big)^*$ with respect to the matrix of weights 
$ W (z)$.

%\item 
\hangindent=.9cm \hangafter=1
{\noindent}$\phantom{ol}${\rm ii)}
The matrix of weights $ W (z)$ satisfies the matrix Pearson equation~\eqref{eq:Pearson} with the boundary conditions~\eqref{eq:bconditions} as well as fulfills the constraint
\begin{gather}
\label{eq:escondida}
\big( {\alpha}^{\mathsf L} - \, {\mathcal M} (h^{\mathsf L}) \big) W = W \big( {\alpha}^{\mathsf R} - \, {\mathcal M} (h^{\mathsf R}) \big).
\end{gather}

%\item 
\hangindent=.9cm \hangafter=1
{\noindent}$\phantom{ol}${\rm iii)}
The matrix of weights $ W (z)$ satisfies the boundary conditions~\eqref{eq:bconditions} as well as 
\begin{gather}
\label{eq:rhoalphal}
{ W}^{\prime \prime} -2 \big( h^{\mathsf L} W\big)^{ \prime} 
+ \alpha^{\mathsf L} W = W \alpha^{\mathsf R} , \\
\label{eq:rhoalphar}
{ W}^{\prime \prime} -2 \big( W h^{\mathsf R} \big)^{ \prime} 
+ W \alpha^{\mathsf R} = \alpha^{\mathsf L} W .
\end{gather}
%and~\eqref{eq:escondida} 
%\end{enumerate}
\end{teo}

\begin{proof}
Following the ideas in the proof of Proposition~\ref{prop:aa}
\begin{gather*}
\langle {\mathcal L}^{\mathsf L} ( P ) , \tilde P \rangle_ W = \langle P , {\mathcal L}^{\mathsf R} (\tilde P ) \rangle_ W
\end{gather*}
if and only if
\begin{gather*}
\langle- P \, {\mathcal M} (h^{\mathsf L}) + P {\alpha}^{\mathsf L} , \tilde P \rangle_ W = \langle P , - \mathcal M (h^{\mathsf R}) \tilde P + {\alpha}^{\mathsf R} \tilde P \rangle_ W
\end{gather*}
that is~\eqref{eq:escondida} takes place, and so i) is equivalent to ii).

To prove that i) is equivalent to iii) observe that, adding~\eqref{eq:rhoalphal} and~\eqref{eq:rhoalphar}, the following~holds
\begin{gather*}
{ W}^{\prime \prime} = \big( h^{\mathsf L} W \big)^{ \prime} +
\big( W h^{\mathsf R} \big)^{ \prime} \, ,
\end{gather*} 
which transforms~\eqref{eq:Pearson} if we integrate requesting boundary conditions~\eqref{eq:bconditions}.
Moreover, if we subtract~\eqref{eq:rhoalphal} and~\eqref{eq:rhoalphar} we arrive directly to~\eqref{eq:escondida}.
\end{proof}

\begin{rem}
 For the symmetric or Hermitian reductions we find that
 \begin{align*}
 \mathcal L^{\mathsf R} (P)&= \big( \mathcal L^{\mathsf L}(P^\top)\big)^\top, & \text{symmetric},\\
 \mathcal L^{\mathsf R} (P)&= \big( \mathcal L^{\mathsf L}(P^\dagger)\big)^\dagger, & \text{Hermitian},
 \end{align*}
 where in the last case we take $x\in\R$.
 \\ 
Moreover, the following are equivalent conditions
 
\begin{enumerate}
 \item 
Equations 
 \begin{align}\label{eq:adjoint_reduced}
\begin{aligned}
 \mathcal L^* (P)&=( \mathcal L(P^\top))^\top,  \qquad \qquad \qquad \text {symmetric},\\
 \mathcal L^* (P)&=( \mathcal L(P^\dagger))^\dagger,  \qquad\qquad\qquad \text {Hermitian};
\end{aligned}
 \end{align}
are satisfied by any matrix polynomial $P$ , where $ \mathcal L:= \mathcal L^{\mathsf L}$.
 
 \item The matrix of weights $ W (z)$ satisfies the matrix Pearson equation~\eqref{eq:Pearson:_reduced} with the boundary conditions
  \begin{align}\label{eq:boundaryc_reduced}
 W|_{\partial\gamma}&=0_N, &\big( W^{\prime} - 2h W\big)\big|_{\partial \gamma}&=0_N, 
 \end{align}
 as well as fulfills the constraint
 \begin{align*}
 \big( {\alpha}- \, {\mathcal M} (h) \big) W &= W \big( {\alpha}^\top - {\mathcal M} (h^\top) \big), & \text{symmetric},\\
 \big( {\alpha}- \, {\mathcal M} (h) \big) W &= W \big( {\alpha}^\dagger - {\mathcal M} ((h(\bar z))^\dagger) \big), & \text{Hermitian},
 \end{align*}

\item 
The matrix of weights $ W (z)$ satisfies the boundary conditions~\eqref{eq:boundaryc_reduced}
as well as 
\begin{align}\label{eq:second_order_reduced}
\begin{aligned}
 { W}^{\prime \prime} -2 \big( h W\big)^{ \prime} 
 + \alpha W &= W \alpha^\top,&\text{symmetric},\\
 { W}^{\prime \prime} -2 \big( h W\big)^{ \prime} 
 + \alpha W &= W \alpha^\dagger,&\text{Hermitian}.
\end{aligned}
 \end{align}
\end{enumerate}
\end{rem}
\subsection{Eigenvalue problems}

% We will see in the next section, for a Hermite matrix generalized situation, that this is not an empty discussion and that in the most simple case there is a good number of examples.

%Observe that according to \eqref{eq:zerocurvature3} and \eqref{eq:zerocurvature4}
%\begin{align*}% \label{eq:zerocurvature3}
%%\big( T_{n}^{\mathsf L}(z) \big)^{\prime}
%\begin{bmatrix}
%M_{1,1,n+1}^\mathsf L+M_{1,1,n}^\mathsf &M_{1,2,n}^\mathsf L\\
%M_{2,1,n+1}^\mathsf L& 0_N
%\end{bmatrix}
%%\big( T_{n}^{\mathsf L}(z) \big)^{\prime}
%&= \big(M^{\mathsf L}_{n+1} (z)\big)^2 
%T^{\mathsf L}_n(z) - T^{\mathsf L}_n (z) \big( M^{\mathsf L}_{n} (z)\big)^2,\\
%%\big(T^{\mathsf R}_{n} (z)\big)^{\prime}
%\begin{bmatrix}
%M_{1,1,n+1}^\mathsf R+M_{1,1,n}^\mathsf &M_{1,2,n+1}^\mathsf R\\
%M_{2,1,n}^\mathsf R& 0_N
%\end{bmatrix}
%%\big(T^{\mathsf R}_{n} (z)\big)^{\prime} 
%&=
%T^{\mathsf R}_n (z) \big(M^{\mathsf R}_{n+1} (z)\big)^2
%- \big( M^{\mathsf R}_{n} (z) \big)^2 T^{\mathsf R}_n (z).
%%\label{eq:zerocurvature4}
%\end{align*}

%\subsubsection{\bf A Duran-Grünbaum approach for Hermite MOPs}
%\paragraph{\textbf{A Duran-Grünbaum approach for Hermite MOPs}}

Now we discuss a result that links our results based on the Riemann--Hilbert problem with previous seminal results by Grünbaum and Durán~\cite{duran_3,duran2004,duran_1,duran_2}.
The next theorem shows when the polynomials and associated functions of second kind are eigenfunctions of a second order operator.

\begin{teo}[Eigenvalue problems for Hermite matrix orthogonal polynomials] \label{teo:nuevo}
Let $h^\mathsf L(z)$ and $h^\mathsf R(z)$ be of degree one matrix polynomials, i.e.
\begin{align*}
h^{\mathsf L}(z) & = A^{\mathsf L} z + B^{\mathsf L} , &
h^{\mathsf R}(z)=A^{\mathsf R} z +B^{\mathsf R} , & &
A^{\mathsf L}, A^{\mathsf R}, B^{\mathsf L}, B^{\mathsf R} \in \mathbb C^{N \times N} ,
\end{align*} 
with $A^{\mathsf L}, A^{\mathsf R}$ definite negative, and $ W(z)$ a matrix of weights a solution of~\eqref{eq:rhoalphal},~\eqref{eq:rhoalphar} subject to the boundary conditions~\eqref{eq:bconditions}.
Then, the following conditions are equivalent:

%\begin{enumerate}
%\item 
\hangindent=.9cm \hangafter=1
{\noindent}$\phantom{ol}${\rm i)} \label{1} 
The operators ${\mathcal L}^{\mathsf L}$ and ${\mathcal L}^{\mathsf R} $ are adjoint operators with respect to the matrix of weights $ W(z)$,
i.e. ${\mathcal L}^{\mathsf R}=\big({\mathcal L}^{\mathsf L}\big)^*$.

%\item 
\hangindent=.9cm \hangafter=1
{\noindent}$\phantom{ol}${\rm ii)} \label{2} 
The biorthogonal polynomial sequences with respect to $ W(z)$, say
$\big\{ P_n^{\mathsf L}(z) \big\}_{n\in\N}$, $ \big\{ P_n^{\mathsf R}(z) \big\}_{n\in\N}$, 
are eigenfunctions of ${\mathcal L}^{\mathsf L}$ and ${\mathcal L}^{\mathsf R}$, i.e. there exists $N \times N $ matrices, $\lambda_n^{\mathsf L}$, $\lambda_n^{\mathsf R}$ such that
\begin{align} \label{eq:eigensystems}
{\mathcal L}^{\mathsf L} (P_n^{\mathsf L})
 &= \lambda_n^{\mathsf L} P_n^{\mathsf L}, &
 {\mathcal L}^{\mathsf R} (P_n^{\mathsf R}) &= P_n^{\mathsf R} \lambda_n^{\mathsf R},
\end{align}
with $\lambda_n^{\mathsf L} C_n^{-1} = C_n^{-1} \lambda_n^{\mathsf R}$, $n \in \mathbb N$.

%\item 
\hangindent=.9cm \hangafter=1
{\noindent}$\phantom{ol}${\rm iii)} \label{3} 
The functions of second kind, $ \big\{ Q_n^\mathsf L(z) \big\}_{n \in \mathbb N}$ and $ \big\{ Q_n^\mathsf R(z) \big\}_{n \in \mathbb N}$, associated with the biorthogonal polynomials, $\big\{ P_n^\mathsf L (z) \big\}_{n \in \mathbb N}$ 
and $\big\{ P_n^\mathsf R (z) \big\}_{n \in \mathbb N}$,
fulfill the second order differential~equations,
\begin{align} \label{eq:eigensystemql}
\big(Q_n^{\mathsf L} \big)^{\prime \prime}(z) 
- 2\big(Q_n^{\mathsf L} \big)^{ \prime}(z) \, h^{\mathsf R} (z)
+ Q_n^{\mathsf L} (z) \, (\alpha^{\mathsf R} - 2 A^{\mathsf R} ) 
 &= \lambda_n^{\mathsf L} \, Q_n^{\mathsf L} (z),
 \\
\label{eq:eigensystemqr}
\big(Q_n^{\mathsf R} \big)^{\prime \prime}(z) 
- 2 h^{\mathsf L} (z) \big(Q_n^{\mathsf R} \big)^{ \prime}(z) 
 +
(\alpha^{\mathsf L} - 2 A^{\mathsf L} ) \, Q_n^{\mathsf R} (z) 
 &= 
Q_n^{\mathsf R} \, \lambda_n^{\mathsf R}.
\end{align}
%\end{enumerate}
\end{teo}

\begin{proof}

 \textbf{ii) implies i).} If $n \not = m$
\begin{gather*}
\langle {\mathcal L}^{\mathsf L} ( P_n^{\mathsf L}(z) ) \, , \, P_m^{\mathsf R}(z) \rangle_ W = 
\lambda_n^{\mathsf L} \langle P_n^{\mathsf L}(z) \, , \, P_m^{\mathsf R}(z) \rangle_ W 
= 0_N \, , \\
\langle P_n^{\mathsf L}(z) \, , \, {\mathcal L}^{\mathsf R} ( P_m^{\mathsf R}(z) ) \rangle_ W
= \langle P_n^{\mathsf L}(z) \, , \, P_m^{\mathsf R}(z) \rangle_ W \lambda_m^{\mathsf R}
= 0_N \, ;
\end{gather*}
and for $n =m$
\begin{align*}
\langle {\mathcal L}^{\mathsf L} ( P_n^{\mathsf L}(z) ) \, , \, P_n^{\mathsf R}(z) \rangle_ W & = \lambda_n^{\mathsf L} C_n^{-1} \, , &
\langle P_n^{\mathsf L}(z) \, , \, {\mathcal L}^{\mathsf R} ( P_n^{\mathsf R}(z) ) \rangle_ W & = C_n^{-1} \lambda_n^{\mathsf R} \, , & n \in \mathbb N,
\end{align*}
which implies that $\langle {\mathcal L}^{\mathsf L} ( P_n^{\mathsf L}(z) ) \, , \, P_m^{\mathsf R}(z) \rangle_ W = \langle P_n^{\mathsf L}(z) \, , \, {\mathcal L}^{\mathsf R} ( P_m^{\mathsf R}(z) ) \rangle_ W $, $n , m \in \mathbb N$.

\textbf{i) implies ii).} Let us note that the space of matrix polynomials of a given degree is invariant under the action of the operators $ \mathcal L^{\mathsf L} $ and $\mathcal L^{\mathsf R} $;
hence
\begin{align*}
{\mathcal L}^{\mathsf L }(P_n^{\mathsf L} )= \sum_{k=0}^n {\lambda}_{n,k}^{\mathsf L }P_k^{\mathsf L} .
\end{align*}
Now, taking into account the biorthogonality of the sequences $P_n^{\mathsf L} $ and $ P_n^{\mathsf R} $ with respect to~$ W$ and using that the operators ${\mathcal L}^{\mathsf L } $ and $ {\mathcal L}^{\mathsf R }$ are adjoint operators we have
\begin{align*}
{\lambda}_{n,k}^{\mathsf L } C_k^{-1} &= \langle {\mathcal L}^{\mathsf L }(P_n^{\mathsf L} ), P_k^{\mathsf R} \rangle_{ W} = 
\langle P_n^{\mathsf L} , {\mathcal L}^{\mathsf R }(P_k^{\mathsf R}) \rangle_{ W} = C_n^{-1} {\lambda}_{n,k}^{\mathsf R }\delta_{n,k}, &
n , m &\in \N ,
\end{align*}
so it holds that
$
{\mathcal L}^{\mathsf L }(P_n^{\mathsf L} )= {\lambda}_{n}^{\mathsf L }P_k^{\mathsf L} 
$
and also 
$
{\mathcal L}^{\mathsf R }(P_n^{\mathsf R} )= {\lambda}_{n}^{\mathsf R }P_k^{\mathsf R} 
$
where $ {\lambda}_{n}^{\mathsf L } C_n^{-1} = C_n^{-1} {\lambda}_{n}^{\mathsf R} $.

\textbf{ ii) implies iii) }We return back to equations~\eqref{eq:secondorderpl} and~\eqref{eq:eigensystems} and see that
\begin{gather*}
\big[ {\mathcal M }(h^{\mathsf L} (z') ) , P_n^{\mathsf L} (z') \big] 
+ \mathsf H^{\mathsf L}_{1,1,n} (z') P_n^{\mathsf L} (z') 
- \mathsf H^{\mathsf L}_{1,2,n} (z') P_{n-1}^{\mathsf L} (z') 
 = - P_n^{\mathsf L} (z') \, \alpha^{\mathsf L} + \lambda_n^{\mathsf L} \, P_n^{\mathsf L} (z') \, .
\end{gather*}
Now, multiplying this equation on the right by $ W (z') / (z-z')$ and integrating along $\gamma$, taking into account the boundary conditions, we get
\begin{multline*}
{\mathcal M }(h^{\mathsf L} (z) ) \, Q_n^{\mathsf L} (z) - Q_n^{\mathsf L} (z) \, {\mathcal M }( - h^{\mathsf R} (z) ) 
+ \mathsf H^{\mathsf L}_{1,1,n} (z) Q_n^{\mathsf L} (z) 
- \mathsf H^{\mathsf L}_{1,2,n} (z) Q_{n-1}^{\mathsf L} (z)
 \\ = Q_n^{\mathsf L} (z) \, (2 A^{\mathsf R} - \alpha^{\mathsf R} ) + \lambda_n^{\mathsf L} \, Q_n^{\mathsf L} (z) \, .
\end{multline*}
Now, from~\eqref{eq:secondorderql} we get~\eqref{eq:eigensystemql}.
%\begin{gather*}
%\left(Q_n^{\mathsf L} \right)^{\prime \prime}(z) 
%- 2\left(Q_n^{\mathsf L} \right)^{ \prime}(z) \, h^{\mathsf R} 
%+ Q_n^{\mathsf L} (z) \, (\alpha^{\mathsf R} - 2 A^{\mathsf R} ) 
% = \lambda_n^{\mathsf L} \, Q_n^{\mathsf L} (z) \, .
%\end{gather*}
We have proved that if $\big\{ P_n^{\mathsf L} \big\}_{n\in\N}$ satisfies a second order linear differential equation the associated functions of second kind also does.

%The key ideas do perform the change of $\{ P_n^{\mathsf L} \}$ into $\{ Q_n^{\mathsf L} \}$ are
 We have that
\begin{align*}
\int_\gamma \frac{\mathcal M (h^{\mathsf L}) (z') } {z'-z} P^{\mathsf L}_n (z') W (z') \d z'
 & = \int_\gamma \frac{(A^{\mathsf L})^2 (z')^2 + \{ A^{\mathsf L} , B^{\mathsf L} \} z' + A^{\mathsf L} + (B^{\mathsf L})^2 } {z'-z} P^{\mathsf L}_n (z') W (z') \d z',
\end{align*}
with the anticommutator notation $\{A,B\}=AB+BA$.
Now, as
\begin{align*}
%(A^{\mathsf L})^2 
\int_\gamma \frac{ (z')^2} {z'-z} P^{\mathsf L}_n (z') W (z') \d z'
 & = %(A^{\mathsf L})^2 \Big( 
 \int_\gamma \frac{ (z')^2 - z^2} {z'-z} P^{\mathsf L}_n (z') W (z') \d z'
+ z^2 Q^{\mathsf L}_n (z) %\Big) 
 \\
 & = %(A^{\mathsf L})^2 \Big( 
 \int_\gamma ({z' + z}) P^{\mathsf L}_n (z') W (z') \d z'
+ z^2 Q^{\mathsf L}_n (z) \, , %\Big)
\end{align*}
and, in the same way,
\begin{align*}
%\{ A^{\mathsf L} , B^{\mathsf L} \}
\int_\gamma \frac{ z' } {z'-z} P^{\mathsf L}_n (z') W (z') \d z'
 & = %\{ A^{\mathsf L} , B^{\mathsf L} \} \Big(
\int_\gamma \frac{ z' -z } {z'-z} P^{\mathsf L}_n (z') W (z') \d z'
+ z Q^{\mathsf L}_n (z) %\Big) 
 \\
& = %\{ A^{\mathsf L} , B^{\mathsf L} \}\Big( 
\int_\gamma P^{\mathsf L}_n (z') W (z') \d z'
+ z Q^{\mathsf L}_n (z) % \Big) 
 \, ,
\end{align*}
we finally obtain
\begin{align*}
\int_\gamma \frac{\mathcal M (h^{\mathsf L}) (z') } {z'-z} P^{\mathsf L}_n (z') W (z') \d z'& = \mathcal M (h^{\mathsf L}) \, Q_n^{\mathsf L} (z), & n \geq 2,
\end{align*}
where we have used the orthogonality conditions for $\big\{ P_n^{\mathsf L} \big\}_{n\in\N}$. 
%(note that
%%$\displaystyle %\frac
%%({w-z})/({z'-z}) = -1$ and 
%$ %\frac
%({ (z')^2 - z^2})/({z'-z}) = - w - z $); 
We also have
\begin{align*}
\int_\gamma P^{\mathsf L}_n (z') \frac{\mathcal M (h^{\mathsf L}) (z') - \alpha^{\mathsf L} } {z'-z} W (z') \d z'
 & = 
\int_\gamma 
P^{\mathsf L}_n (z') W (z') \frac{\mathcal M (h^{\mathsf R} ) (z') - \alpha^{\mathsf R} }{z'-z} \d z'%
 \\ 
 & = 
Q_n^{\mathsf L} (z) \, (\mathcal M (h^{\mathsf R}) (z) - \alpha^{\mathsf R})
, & n \geq 2 .
\end{align*}
Using the same ideas we prove that
%\textbf{Por qué??? }
\begin{align} \label{eq:igualdadepolinomial}
\int_\gamma \frac{\mathsf H^{\mathsf L}_{1,j,n} (z')}{z'-z} P_{n-j+1}^{\mathsf L} (z') W (z') \d z'& = \mathsf H^{\mathsf L}_{1,j,n} (z) Q_{n-j+1}^{\mathsf L} (z),& n &\geq 1, \, j = 1,2 \, .
%\int_\gamma \frac{\mathsf H^{\mathsf L}_{1,2,n} (z')}{z'-z} P_{n-1}^{\mathsf L} (z') W (z') \, \d z'= \mathsf H^{\mathsf L}_{1,2,n} (z) Q_{n-1}^{\mathsf L} (z)
\end{align}
In fact, by definition~\eqref{eq:Miura_bispectrality_L} we know that the matrix polynomials $\mathsf H^{\mathsf L}_{1,j,n} (z')$ are of degree at most one, i.e.
\begin{gather*}
\mathsf H^{\mathsf L}_{1,j,n} (z') = \mathsf H^{\mathsf L,0}_{1,j,n} z' + \mathsf H^{\mathsf L,1}_{1,j,n} \, , \ \ \
\mathsf H^{\mathsf L,0}_{1,j,n} , \mathsf H^{\mathsf L,1}_{1,j,n} \in \mathbb C^{N \times N} \, .
\end{gather*}
Summing and subtracting in~\eqref{eq:igualdadepolinomial} $\mathsf H^{\mathsf L}_{1,j,n} (z)$ we get in the left hand side
\begin{multline*}
\int_\gamma \frac{\mathsf H^{\mathsf L}_{1,j,n} (z')}{z'-z} P_{n-j+1}^{\mathsf L} (z') W (z') \d z'\\
 = 
\int_\gamma \frac{\mathsf H^{\mathsf L}_{1,j,n} (z') - \mathsf H^{\mathsf L}_{1,j,n} (z)}{z'-z} P_{n-j+1}^{\mathsf L} (z') W (z') \d z'
+ \mathsf H^{\mathsf L}_{1,j,n} (z) Q_{n-j+1}^{\mathsf L} (z) \, ; 
\end{multline*}
hence, as
\begin{gather*}
\frac{\mathsf H^{\mathsf L}_{1,j,n} (z') - \mathsf H^{\mathsf L}_{1,j,n} (z)}{z'-z}
= \mathsf H^{\mathsf L,0}_{1,j,n} \, ,
\end{gather*}
we arrive to
\begin{gather*}
\int_\gamma \frac{\mathsf H^{\mathsf L}_{1,j,n} (z')}{z'-z} P_{n-j+1}^{\mathsf L} (z') W (z') \d z'
 = 
\mathsf H^{\mathsf L,0}_{1,j,n}
\int_\gamma P_{n-j+1}^{\mathsf L} (z') W (z') \d z'
+ \mathsf H^{\mathsf L}_{1,j,n} (z) Q_{n-j+1}^{\mathsf L} (z) \, ,
\end{gather*}
and by the orthogonality of $\big\{ P_{n-j+1}^{\mathsf L} (z) \big\}_{n \in \N} $ with respect to $ W(z)$ we get for $j =1,2$, and for all $n = 1, 2, \ldots $, that~\eqref{eq:igualdadepolinomial} holds true.

From~\eqref{eq:secondorderpr} and taking into account that $\mathcal L^{\mathsf R} (P_n^{\mathsf R}) = P_n^{\mathsf R} \, \lambda_n^{\mathsf R} $ we get
\begin{gather*} 
\big[ P_n^{\mathsf R} (z') , \mathcal M(h^{\mathsf R}) (z') \big]
%- \mathcal M(h^{\mathsf R} (z') ) P_n^{\mathsf R} (z')
+ P_n^{\mathsf R} (z') \mathsf H^{\mathsf R}_{1,1,n} (z')
- P_{n-1}^{\mathsf R} (z') \mathsf H^{\mathsf R}_{2,1,n} (z')
 = - \alpha^{\mathsf R} \, P_n^{\mathsf R} (z') + P_n^{\mathsf R} (z') \, \lambda_n^{\mathsf R} \, .
\end{gather*}
Now, multiplying this equation on the left by $ W (z') / (z-z')$ and integrate (using the boundary conditions) over $\gamma$, we get
\begin{multline*} 
Q_n^{\mathsf R} (z) \, \mathcal M(h^{\mathsf R} ) (z)
- \mathcal M(-h^{\mathsf L} ) (z) Q_n^{\mathsf R} (z)
+ Q_n^{\mathsf R} (z) \mathsf H^{\mathsf R}_{1,1,n} (z)
- Q_{n-1}^{\mathsf R} (z) \mathsf H^{\mathsf R}_{2,1,n} (z)
 \\ 
 = (2 A^{\mathsf L} - \alpha^{\mathsf L}) \, Q_n^{\mathsf R} + Q_n^{\mathsf R} \, \lambda_n^{\mathsf R} \, ,
\end{multline*}
and so, from~\eqref{eq:secondorderqr} we arrive to~\eqref{eq:eigensystemqr}.
%\begin{gather*} 
%\left(Q_n^{\mathsf R} \right)^{\prime \prime}(z) - 2 h^{\mathsf L} \left(Q_n^{\mathsf R} \right)^{ \prime}(z) 
% +
%(\alpha^{\mathsf L} - 2 A^{\mathsf L} ) \, Q_n^{\mathsf R} (z) 
% = 
%Q_n^{\mathsf R} \, \lambda_n^{\mathsf R} \, .
%\end{gather*}

\textbf{ iii) implies ii)}. 
%From the definition of 
%\begin{gather*}
%Q_n^\mathsf L (z) = \int_\gamma \frac{P_n^\mathsf L (z') W (z')}{z -w} \d z'
%\end{gather*}
%and 
Taking derivatives with respect to $z$ we get, after integration by parts and using the boundary conditions
\begin{align*}
 (Q_n^\mathsf L)^{\prime} (z)& = \int_\gamma \frac{P_n^\mathsf L (z') W (z')}{(z' -z)^2} \d z'
%= - \int_\gamma \frac{(P_n^\mathsf L (z') W (z'))^{\prime}}{z -w} \d z'
 \, , \\
%\end{align*}
%\begin{align*} 
(Q_n^\mathsf L)^{\prime\prime} (z) &= 2 \int_\gamma \frac{P_n^\mathsf L (z') W (z') }{(z' -z)^3} \d z'= \int_\gamma \frac{(P_n^\mathsf L (z') W (z'))^{\prime\prime}}{z' -z} \d z'.
\end{align*}
Moreover,
\begin{align*}
-2 (Q_n^\mathsf L)^{\prime} (z) h^{\mathsf R}(z) 
 & = 
2 \int_\gamma P_n^\mathsf L (z') W (z') \frac{h^{\mathsf R} (z') - h^{\mathsf R} (z)}{(z' -z)^2} \d z'
- 2 \int_\gamma P_n^\mathsf L (z') W (z') 
\frac{ h^{\mathsf R} (z')}{(z' -z)^2} \d z'\\
 & = 2 Q_n^\mathsf L (z) A^{\mathsf R} - 2 \int_\gamma \frac{(P_n^\mathsf L (z') W (z') h^{\mathsf R} (z') )^{\prime}}{z' -z} \d z'.
\end{align*}
Now, we plug all this information into~\eqref{eq:eigensystemql} and deduce that
\begin{align*}
\int_\gamma \frac{(P_n^{\mathsf L})^{\prime\prime} 
 % (z') 
 W
+ 2 (P_n^{\mathsf L})^{\prime} 
 % (z') 
( W^{\prime} - W h^{\mathsf R} ) + P_n^{\mathsf L} 
 % (z') 
( W^{\prime\prime} - 2 ( W h^{\mathsf R})^{\prime}) + W \alpha^{\mathsf R}) }{z'-z} \, \d z'
 = \lambda_n^\mathsf L \int_\gamma \frac{ P_n^{\mathsf L}
 % (z')
 W 
 % (z')
}{z'-z} 
\d z'\, ;
\end{align*}
by the hypothesis over $ W$ we get
\begin{gather*}
\int_\gamma \frac{(P_n^{\mathsf L})^{\prime\prime} (z')+ 2 (P_n^{\mathsf L})^{\prime} (z') h^{\mathsf L} (z') + P_n^{\mathsf L} \alpha^\mathsf L - \lambda_n^\mathsf L P_n^{\mathsf L} }{z'-z} \, W (z') \, \d z'
= 0_N.
\end{gather*}
%where the function in the left hand side is analytic in a neighborhood of $\infty$. 
Hence, we get that $\big\{ P_n^\mathsf L \big\}_{n \in \N}$ satisfies~\eqref{eq:eigensystems}.
Using analogous arguments it can be proven that the equation~\eqref{eq:eigensystemqr} for $\big\{ Q_n^{\mathsf R} \big\}_{n\in\N} $ implies that $\big\{ P_n^\mathsf R \big\}_{n\in\N}$ satisfies~\eqref{eq:eigensystems}.
\end{proof}

%\textbf{Aquí falta un parrafo comparando este resultado con los de Duran y Grünbaum}

The interpretation in terms of adjoint operators, inherits from the Riemann--Hilbert problem the characterization for the $\big\{Q_n^\mathsf L \big\}_{n \in \N}$ and 
$\big\{Q_n^\mathsf R \big\}_{n \in \N}$. Moreover, Theorems~\ref{teo:novo} we see that $W$ in Theorem~\ref{teo:nuevo} can be taken as a solution of a Pearson Sylvester differential equation like~\eqref{eq:Pearson} and satisfies~\eqref{eq:escondida}.

\begin{rem} For the symmetric or Hermitian reductions we take
$ h(z) = A z + B, $
 with $A$ definite negative, and $ W(z)$ a matrix of weights a solution of~\eqref{eq:second_order_reduced} subject to the boundary conditions~\eqref{eq:boundaryc_reduced}.
 Then, the following conditions are equivalent:
 
 %\begin{enumerate}
 %\item 
 \hangindent=.9cm \hangafter=1
 {\noindent}$\phantom{ol}${\rm i)} \label{1} 
Equation~\eqref{eq:adjoint_reduced} is satisfied.
 
 %\item 
 \hangindent=.9cm \hangafter=1
 {\noindent}$\phantom{ol}${\rm ii)} \label{2} 
 The matrix orthogonal polynomials with respect to $ W(z)$
%  say
% $\{ P_n(z) \}_{n\in\N}$, 
 are eigenfunctions of ${\mathcal L}$.
 
 %\item 
\hangindent=.9cm \hangafter=1
{\noindent}$\phantom{ol}${\rm iii)} \label{3} 
The functions of second kind, $ \big\{ Q_n(z) \big\}_{n \in \mathbb N}$, associated with the matrix orthogonal polynomials, $\big\{ P_n (z) \big\}_{n \in \mathbb N}$ fulfill the second order differential~equations,
 \begin{align*}
 \big(Q_n\big)^{\prime \prime}(z) 
 - 2\big(Q_n\big)^{ \prime}(z) \, (h (z))^\top
 + Q_n (z) \, (\alpha^\top - 2 A^\top ) 
 &= \lambda_n\, Q_n (z), &\text{symmetric},\\
  \big(Q_n\big)^{\prime \prime}(z) 
 - 2\big(Q_n\big)^{ \prime}(z) \, (h (\bar z))^\dagger
 + Q_n (z) \, (\alpha^\dagger - 2 A^\dagger ) 
 &= \lambda_n\, Q_n (z), &\text{Hermitian}.
 \end{align*}
 %\end{enumerate}
 
\end{rem}

The equivalences, described in the previous remark, excluding the one for the second kind functions (which is new), coincide with those of~\cite{duran2004}.
Therefore, these results could understood as an extension of those by Durán and Grünbaum to the non Hermitian orthogonality scenario.

\section{Nonlinear %equations 
difference equations for the recursion coefficients} \label{sec:6}

Using the Riemann--Hilbert approach we will derive in this section nonlinear matrix difference equations fullfilled by the recursion coefficients. We will consider three different possibilities for the Pearson equations satisfied by the matrix of weights.

\subsection{Nonlinear difference equations for Hermite matrix polynomials}

We now explore the most simplest case when
%with 
$\max(h^{\mathsf L}_n (z) ,h^{\mathsf R}_n (z) )=1$ in full generality.
We take 
\begin{align*}
h^{\mathsf L}(z)&= A^{\mathsf L} z + B^{\mathsf L}, &
h^{\mathsf R}(z)=A^{\mathsf R} z +B^{\mathsf R},
\end{align*}
for arbitrary matrices $A^{\mathsf L}, B^{\mathsf L}, A^{\mathsf R}, B^{\mathsf R}\in\C^{N\times N}$, with $A^{\mathsf L}, A^{\mathsf R}$ definite negative matrices.
Thus, the matrix of weights $ W(z)$ is a solution of the following Pearson equation (a Sylvester linear differential equation)
\begin{align*}
 W^{\prime} (z)= (A^{\mathsf L} z + B^{\mathsf L}) W(z) + W(z) (A^{\mathsf R} z + B^{\mathsf R}).
\end{align*}
For simplicity we take $\gamma=\R$.
Hence, the structure matrices have, cf.~\eqref{eq:MSP} and~\eqref{eq:MSPR}, the following form
\begin{align}\label{eq:structure_Hermite}
%\begin{aligned}
M^\mathsf L_n (z) & = {\mathcal A}^\mathsf L z + \mathcal K_n^\mathsf L, &
\mathcal A^\mathsf L & 
 = 
\left[ \begin{smallmatrix}
A^{\mathsf L} & 0_N \\
0_N & -A^{\mathsf R} \end{smallmatrix} \right]
, &
\mathcal K_n^\mathsf L 
& 
 =
\left[ \begin{smallmatrix}
B^{\mathsf L} + \big[ p_{\mathsf L , n}^1 , A^{\mathsf L} \big] 
 & C_n^{-1} A^{\mathsf R} + A^{\mathsf L} C_n^{-1} \\
- C_{n-1} A^{\mathsf L} - A^{\mathsf R} C_{n-1} 
 & -B^{\mathsf R} -\big[ q_{\mathsf L , n-1}^1 , A^{\mathsf R} \big]
\end{smallmatrix} \right] ,
%\\
%\label{eq:structure_HermiteR}
%M^\mathsf R_n (z)&= {\mathcal A}^\mathsf R z + \mathcal K_n^\mathsf R, &
%\mathcal A^\mathsf R & = \left[ \begin{smallmatrix} 
%A^{\mathsf R} & 0_N \\
%0_N & -A^{\mathsf L} \end{smallmatrix} \right] ,&
%\mathcal K_n^\mathsf R 
%& = \left[ \begin{smallmatrix}
%B^{\mathsf R} + \big[ p_n^1 , A^{\mathsf R} \big] & -C_{n-1} A^{\mathsf L} - A^{\mathsf R} C_{n-1}\\
%C_{n}^{-1} A^{\mathsf R} + A^{\mathsf L} C_{n}^{-1}& -B^{\mathsf L} -\big[ q_{n-1}^1 , A^\mathsf L \big]
%\end{smallmatrix} \right] .
%\end{aligned}
\end{align}
%\begin{align*}
%M^\mathsf L_n (z) =
% \begin{bmatrix}
%A^{\mathsf L} z + B^{\mathsf L} & 0_N \\
%0_N & A^{\mathsf R} z + B^{\mathsf R}
%\end{bmatrix} 
%+
%\begin{bmatrix}
%\big[ p_n^1 , A^{\mathsf L} \big] & C_n^{-1} A^{\mathsf R} - A^{\mathsf L} C_n^{-1} \\
%C_{n-1} A^{\mathsf L} - A^{\mathsf R} C_{n-1} & \big[ q_{n-1}^1 , A^{\mathsf R} \big]
%\end{bmatrix} 
%\end{align*}
%Here, the left structure matrix reads
%
%\begin{align}\label{eq:structure_Hermite}
%\begin{aligned}
%M^\mathsf L_n (z) & = {\mathcal A}^\mathsf L z + \mathcal K_n^\mathsf L, &
%\mathcal A^\mathsf L & 
% = 
%\left[ \begin{smallmatrix}
%A^{\mathsf L} & 0_N \\
%0_N & -A^{\mathsf R} \end{smallmatrix} \right]
%, &
%\mathcal K_n^\mathsf L 
%& 
% =
%\left[ \begin{smallmatrix}
%B^{\mathsf L} + \big[ p_n^1 , A^{\mathsf L} \big] & -C_n^{-1} A^{\mathsf R} - A^{\mathsf L} C_n^{-1} \\
%C_{n-1} A^{\mathsf L} + A^{\mathsf R} C_{n-1} & -B^{\mathsf R} -\big[ q_{n-1}^1 , A^{\mathsf R} \big]
%\end{smallmatrix} \right] ,\\
%%\label{eq:structure_HermiteR}
%M^\mathsf R_n (z)&= {\mathcal A}^\mathsf R z + \mathcal K_n^\mathsf R, &
%\mathcal A^\mathsf R & = \left[ \begin{smallmatrix} 
%A^{\mathsf R} & 0_N \\
%0_N & -A^{\mathsf L} \end{smallmatrix} \right] ,&
%\mathcal K_n^\mathsf R 
%& = \left[ \begin{smallmatrix}
%B^{\mathsf R} - \big[ p_n^1 , A^{\mathsf R} \big] & C_{n-1} A^{\mathsf L} + A^{\mathsf R} C_{n-1}\\
%-C_{n}^{-1} A^{\mathsf R} - A^{\mathsf L} C_{n}^{-1}& -B^{\mathsf L} +\big[ q_{n-1}^1 , A^\mathsf L \big]
%\end{smallmatrix} \right] .
%\end{aligned}
%\end{align}
%\subsubsection{Nonlinear equations for the recursion coefficients}
The Silvester differential system~\eqref{eq:estrutura1} for the left fundamental matrix is
\begin{align*}
\big(Y^\mathsf L_n(z)\big)^{\prime}+ 
\left[ Y^\mathsf L_n (z) , 
\left[\begin{smallmatrix}
A^{\mathsf L} z + B^{\mathsf L} & 0_N \\
0_N & -A^{\mathsf R} z - B^{\mathsf R}
\end{smallmatrix} \right] \right]
 & =
\left[ \begin{smallmatrix}
\big[ p_{\mathsf L , n}^1 , A^{\mathsf L} \big] & C_n^{-1} A^{\mathsf R} + A^{\mathsf L} C_n^{-1} \\
-C_{n-1} A^{\mathsf L} - A^{\mathsf R} C_{n-1} & -\big[ q_{\mathsf L , n-1}^1 , A^{\mathsf R} \big]
\end{smallmatrix} \right] \, Y_n (z), & n \in \mathbb N ,
\end{align*}
that is, for all $n \in \mathbb N$,
\begin{gather}
\label{eq:Hermite1}
(P^{\mathsf L}_n)^{\prime} + \Big[ P^{\mathsf L}_n , A^{\mathsf L} z + B^{\mathsf L} \big]
=
\big[ p_{\mathsf L,n}^1 , A^{\mathsf L} \big] P^{\mathsf L}_n - \big( C_n^{-1} A^{\mathsf R} +A^{\mathsf L} C_n^{-1} \big) C_{n-1} P^{\mathsf L}_{n-1} ,\\ \label{eq:Hermite2}
C_{n-1} (Q^{\mathsf L}_{n-1})^{\prime} - \Big[ C_{n-1} Q^{\mathsf L}_{n-1} , A^{\mathsf R} z + B^{\mathsf R} \big]
=
\big( C_{n-1} A^{\mathsf L} +A^{\mathsf R} C_{n-1} \big) Q^{\mathsf L}_n - \big[ q_{n-1}^1 , A^{\mathsf R} \big] C_{n-1} Q^{\mathsf L}_{n-1},
\\
\label{eq:Hermite3}
\begin{multlined}[t][.85\textwidth]
C_{n-1} (P^{\mathsf L}_{n-1})^{\prime} + C_{n-1} P_{n-1} \big( A^{\mathsf L} z + B^{\mathsf L} \big) +\big( A^{\mathsf R} z + B^{\mathsf R} \big) \, C_{n-1} P_{n-1}^{\mathsf L}
\\ =
\big( C_{n-1} A^{\mathsf L} +A^{\mathsf R} C_{n-1} \big) P^{\mathsf L}_{n} -
\big[ q_{{\mathsf L},n-1}^1 , A^{\mathsf R} \big] C_{n-1} P^{\mathsf L}_{n-1}, 
\end{multlined}
\\ \label{eq:Hermite4}
(Q^{\mathsf L}_{n})^{\prime} - Q^{\mathsf L}_{n} \big( A^{\mathsf R} z + B^{\mathsf R} \big) - \big( A^{\mathsf L} z + B^{\mathsf L} \big) Q^{\mathsf L}_{n} =
\big[ p_{n}^1 , A^{\mathsf L} \big] Q^{\mathsf L}_{n} -
\big( C_n^{-1} A^{\mathsf R} +A^{\mathsf L} C_n^{-1} \big) C_{n-1} Q^{\mathsf L}_{n-1} .
\end{gather}
Taking the $(n-1)$-th $z$ power of the~\eqref{eq:Hermite1}, the $-n$-th of~\eqref{eq:Hermite2}, the $-(n-1)$-th of~\eqref{eq:Hermite3} and the $-(n+1)$-th of~\eqref{eq:Hermite4} we get, for all $n \in \mathbb N$,
\begin{align*} 
n I_N + \big[ p_{\mathsf L,n}^1 , B^{\mathsf L} \big] + \big[ p_{\mathsf L,n}^2 , A^{\mathsf L} \big]
&=
\big[ p_{\mathsf L,n}^1 , A^{\mathsf L} \big] \, p_n^1 - \big( C_n^{-1} A^{\mathsf R} + A^{\mathsf L} C_n^{-1} \big) C_{n-1} ,
 \\
n I_N + \big[ q_{n-1}^1 , B^{\mathsf R} \big] + \big[ q_{\mathsf L,n-1}^2 , A^{\mathsf R} \big]
&=-
\big( C_{n-1} A^{\mathsf L} +A^{\mathsf R} C_{n-1} \big) C_{n}^{-1}
+\big[ q_{\mathsf L,n-1}^1 , A^{\mathsf R} \big] q_{\mathsf L,n-1}^1 ,
 \\
%\end{align*}
%\begin{align*}
C_{n-1} B^{\mathsf L} +B^{\mathsf R} C_{n-1}
+ C_{n-1} \big[ p_{\mathsf L,n-1}^1 , A^{\mathsf L} \big]
&=
- \big( C_{n-1} A^{\mathsf L} +A^{\mathsf R} C_{n-1} \big) \beta^\mathsf L_{n-1}
-
\big[ q_{\mathsf L,n-1}^1 , A^{\mathsf R} \big] C_{n-1},
 \\
B^{\mathsf R} C_{n} +C_{n} B^{\mathsf L}
+ \big[ q_{\mathsf L,n}^1 , A^{\mathsf R} \big] C_{n}
&=
-C_{n} \big[ p_{\mathsf L,n}^1 , A^{\mathsf L} \big]
- \big( A^{\mathsf R} C_{n} + C_{n} A^{\mathsf L} \big) \beta^\mathsf L_{n}.
\end{align*}
After some cleaning we reckon that the system is, for all $n \in \mathbb N$, equivalent to
\begin{gather*}
\begin{cases}
\begin{multlined}
I- \Bigg[ \beta^\mathsf L_n , B^{\mathsf L} - \Big[ \sum_{k=0}^{n-1}\beta^\mathsf L_k , A^{\mathsf L} \Big] + A^{\mathsf L} \beta^\mathsf L_n \Bigg]
\\ \phantom{olaolaolaolaolaola} =
C_{n}^{-1} C_{n-1} A^{\mathsf L}
- C_{n+1}^{-1} A^{\mathsf R} C_{n}
- A^{\mathsf L} C_{n+1}^{-1} C_{n}
+ C_n^{-1} A^{\mathsf R} \, C_{n-1},
\end{multlined}
\\
\begin{multlined}
C_{n-1} B^{\mathsf L} +B^{\mathsf R} C_{n-1}
- C_{n-1} \Big[ \sum_{k=0}^{n-2}\beta^\mathsf L_k , A^{\mathsf L} \Big]
\\ \phantom{olaolaolaolaol} =
- \big( C_{n-1} A^{\mathsf L} +A^{\mathsf R} C_{n-1} \big) \beta^\mathsf L_{n-1}
-
\Big[ \sum_{k=0}^{n-1}C_k\beta^\mathsf L_k(C_k)^{-1} , A^{\mathsf R} \Big] C_{n-1}.
\end{multlined}
\end{cases}
\end{gather*}

\subsection{A matrix extension of the alt-dPI}

We now discuss the case
%with 
$\max(h^{\mathsf L}_n (z) ,h^{\mathsf R}_n (z) )=2$, but we perform a strong simplification as
we take $h^\mathsf R=0_N$ and $ h^\mathsf L= \lambda + \mu z + \nu z^2 $, with $\lambda,\mu,\nu\in\C^{N\times N}$ arbitrary matrices but for $\nu$ being negative definite nonsingular matrix.
Thus, the Pearson equation will be
\begin{align}\label{eq:Pearson_altdPI}
 W^{\prime} (z) = (\lambda + \mu z + \nu z^2 ) W(z).
\end{align} 
We obviously drop off the notation that distinguish left and right polynomials and only describe the results for the left case. The integrals are taken along $\gamma$, a smooth curve for which we have a \emph{simple} Riemann--Hilbert problem as depicted in the following diagram:
\begin{center}
\begin{tikzpicture}
\begin{axis}[axis lines=middle,grid=both, width=0.69\textwidth
,axis equal,
 xmax=2,ymin=-4,
 xmin=-.1, ymax=4,
 xticklabel,yticklabel,disabledatascaling,xlabel=$x$,ylabel=$y$,every axis x label/.style={
  at={(ticklabel* cs:1)},
  anchor=south west,
 },
 every axis y label/.style={
  at={(ticklabel* cs:1.0)},
  anchor=south west,
 },grid style={line width=.1pt, draw=Bittersweet!10},
 major grid style={line width=.2pt,draw=Bittersweet!50},
 minor tick num=4,
axis line style={latex'-latex'},Bittersweet] 
\node[anchor = north east,Bittersweet] at (axis cs: 5,4) {$\mathbb C$} ;
\addplot[NavyBlue,samples=80, ultra thick,domain=-2:2,decoration={
  markings,
  mark=between positions .4 and .5 step 10em with {\arrow [scale=1.1]{latex'}}
 }, postaction=decorate]({cosh(x)},{sqrt(3)*sinh(x)}) ;%
 \addplot[NavyBlue,samples=80, thick,domain=0:3,dashed]({(x)},{sqrt(3)*(x)}) ;
 \addplot[NavyBlue,samples=80, thick,domain=0:3,dashed]({(x)},{-sqrt(3)*(x)}) ;
 \draw[->,>=latex',semithick] (0:1cm) arc (0:60:1cm) node[pos=0.5,sloped,above]{$\frac{\pi}{3}$};
\draw[thick,black]  (axis cs:2,3) node[right ] {$\gamma$} ;
\end{axis}
\draw (5.5,-0.5) node
{\begin{minipage}{.85\textwidth} 
\begin{center}
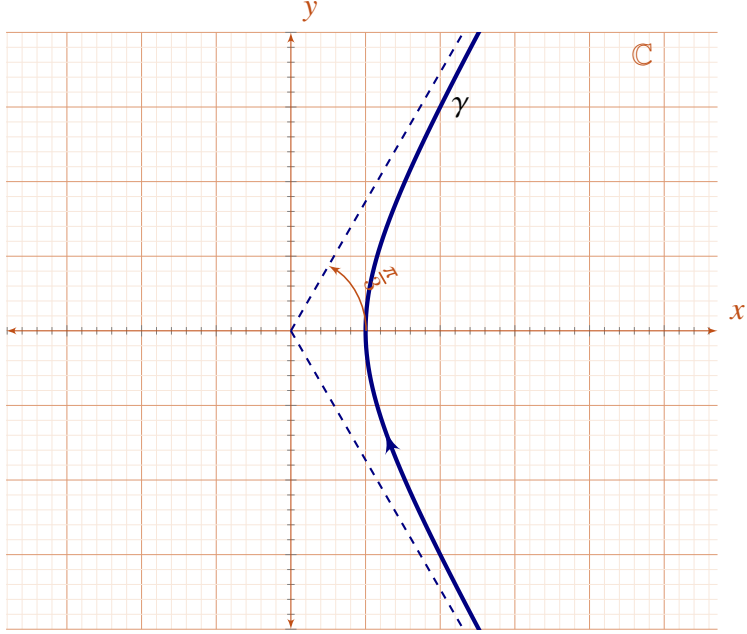
{\bfseries \small Branch of hiperbola $3x^2-y^2=3$$\phantom{olaolaol}$}
\end{center}
\end{minipage}
};
\end{tikzpicture}
\end{center}
The structure matrix, cf.~\eqref{eq:MSP}, is a second order polynomial $M_n (z) = M_n^0 z^2 + M_n^1 z + M_n^2$~with
\begin{align*}
M_n^0 &= 
\begin{bmatrix} 
\nu & 0_N \\ 
0_N & 0_N
\end{bmatrix}, 
\phantom{olaolaola} %&
M_n^1 %&
 = 
\begin{bmatrix} 
\mu - \big[ \nu , p_n^1 \big] & \nu C_n^{-1} \\[.05cm] 
- C_{n-1} \nu & \underline{0} 
\end{bmatrix}, \\
%\end{align*}
%and
%\begin{align*}
M_n^2 &= 
\begin{bmatrix} 
\lambda - \big[ \beta , p_n^1 \big] - \big[ \nu , p_n^2 \big] + \nu \big( p_n^1 \big)^2 - p_n^1 \nu \, p_n^1 + \nu C_n^{-1} C_{n-1} & \big( \mu - \big[ \nu , p_n^1 \big] + \gamma \beta_n \big) C_n^{-1} \\[.05cm] 
- C_{n-1} \, \big( \mu + p_{n-1}^1 \nu - \nu p_{n}^1 \big) & 
- C_{n-1} \nu \, C_n^{-1}
\end{bmatrix} .
\end{align*}

\begin{pro}[Matrix alt-dPI system]
The recursion coefficients $\beta_n,\gamma_n$ of the matrix orthogonal polynomials with matrix of weights a solution of the Pearson equation~\eqref{eq:Pearson_altdPI} are subject to the following system of equation, for all $n \in \mathbb N$,
\begin{gather} \label{eq:alt-dPI_1}
\Big( \mu + \Big[ \nu , \sum\limits_{k=0}^{n-1} \beta_k \Big] + 
\gamma (\beta_n + \beta_{n+1} ) \Big) \gamma_{n+1}=-(n+1) I, \\
%\end{gather}
%\begin{gather}
\label{eq:alt-dPI_2}
\begin{multlined}[t][.85\textwidth]
\lambda + \gamma \big( \gamma_n + \gamma_{n+1} + \beta_n^2 \big) - \mu \beta_n + \Big[ \mu , \sum\limits_{k=0}^{n-1}\beta_k \Big] \big(I_N+ \beta_n\big) \\+ \Big[ \nu , \sum\limits_{m=1}^{n-1}\gamma_m-\sum\limits_{
0\leq k<m\leq n-1}\beta_m\beta_k \Big] 
+ \Big[ \nu , \sum\limits_{k=0}^{n-1}\beta_k \Big] \sum\limits_{k=0}^{n-1}\beta_k = 0_N .
\end{multlined}
\end{gather}
\end{pro}
\begin{proof}
Given the asymptotics about $\infty $, 
\begin{align*}
-C_n Q_n (z) 
=I_Nz^{-n-1}+ q_n^1 z^{-n-2} + \cdots ,
\end{align*} 
we read the coefficient of $z^{-n-1}$ coming from 
\begin{gather*}
C_{n-1} Q_{n-1}^{\prime} (z) 
= - M_{2,1}^n (z) Q_n (z) + M_{2,2}^n (z) \, C_{n-1} Q_{n-1} (z) \, , 
\end{gather*}
with 
%\begin{gather*}
$ M_{2,1}^n = - C_{n-1}\nu z - C_{n-1} \big( \mu + p_{n-1}^1 \nu - \nu p_{n}^1 \big) \, , \ \ \
 %$ and $ 
M_{2,2}^n = - C_{n-1} \nu C_n^{-1} \, $,
%\end{gather*} 
we get~\eqref{eq:alt-dPI_1}; and
from
\begin{gather*}
Q_n^{\prime}(z) 
= M_{1,1}^n \, Q_n (z) - M_{1,2}^n (z) \, C_{n-1} \, Q_{n-1} (z) \, , 
\end{gather*}
with 
\begin{gather*}
M_{1,1}^n = \nu z^2 + \big( \mu - \big[ \nu , p_n^1 \big] \big) z + \big( \lambda - \big[ \mu , p_n^1 \big] - \big[ \nu , p_n^2 \big] + \nu \big( p_n^1 \big)^2 
+ \nu C_n^{-1} C_{n-1} - p_n^1 \nu \, p_n^1 \big) \\
%$ and $ 
M_{1,2}^n = \nu C_n^{-1} z + \big( \mu - \big[ \nu , p_n^1 \big] + \nu \beta_n \big) C_n^{-1} \, ;
\end{gather*}
we deduce~\eqref{eq:alt-dPI_2} from the $ z^{-n-1} $-coefficient.
\end{proof}

Another form of writing this result is 
\begin{pro}[Matrix alt-dPI system]
Given matrix orthogonal polynomials with matrix of weights $ W(z)$ supported on $\gamma$, a solution of the Pearson equation~\eqref{eq:Pearson_altdPI}, the recursion coefficients $\gamma_n$ can be expressed directly in terms of the recursion coefficients $\beta_n$, for all $n \in \mathbb N$,
\begin{align*}
\gamma_{n+1}=-(n+1) \Big( \beta + \Big[ \gamma , \sum\limits_{k=0}^{n-1} \beta_k \Big] + \gamma (\beta_n + \beta_{n+1} ) \Big)^{-1}.
\end{align*}
The coefficients $\beta_n$ fulfill, for all $n \in \mathbb N$, the following non-Abelian alt-dPI,
 \begin{multline*}
\lambda + \nu \big( \gamma_n + \gamma_{n+1} + \beta_n^2 \big) - \mu \beta_n + \Big[ \beta , \sum\limits_{k=0}^{n-1}\beta_k \Big] \big(I_N+ \beta_n\big) \\+ \big[ \nu , \sum\limits_{m=1}^{n-1}\gamma_m-\sum\limits_{
 0\leq k<m\leq n-1}\beta_m\beta_k \big] 
+ \Big[ \nu , \sum\limits_{k=0}^{n-1}\beta_k \Big] \sum\limits_{k=0}^{n-1}\beta_k = 0_N .
\end{multline*}
\end{pro}

\begin{proof}
From~\eqref{eq:alt-dPI_1} we get the $\gamma_{n}$ in terms of $\beta_n$, 
plugged this relation into the second one gives the following nonlinear equation for the matrices $\beta_n$.
\end{proof}

If we assume that $\nu =-I$ as expected strong simplifications occur.
In the first place we find that
\begin{align*}
 \gamma_{n+1}=-(n+1)(\mu - \beta_n - \beta_{n+1})^{-1} ,
%)^{-1}.
\end{align*}
and, secondly, we derive the following simplified version of a non-Abelian alt-dPI equation
 \begin{align*}
\lambda - \beta_n^2
+ n (\beta - \beta_{n-1} +\beta_{n})^{-1}
 %)^{-1}
+ (n+1)(\mu - \beta_n - \beta_{n+1})^{-1} 
%)^{-1} 
- \mu \beta_n
=-\Big[ \mu , \sum\limits_{k=0}^{n-1}\beta_k \Big] \big(I_N+ \beta_n\big) .
 \end{align*}
Moreover, when we choose $\nu=-I$ and $\mu=0_N$ the non local terms disappear and the equation simplifies further to 
\begin{align*}
  - n(\beta_{n-1} + \beta_{n} )^{-1}
%)^{-1}
 -(n+1)(\beta_n + \beta_{n+1} )^{-1}
% )^{-1}
 +\beta_n^2=\lambda.
\end{align*}
Let us remind the reader how the alt-dPI equation appeared for the first time. Going back to the scalar context, in Magnus' work~\cite{magnus:_1}, associated with the weight functions solution of the Pearson equation
$ W^{\prime} (z) = \big( z^2 + t \big) W (z)$, we can find the following {scalar alternate discrete Painlevé I } system
\begin{align*}
\gamma_n + \gamma_{n+1} + \beta_n^2 + t &= 0, \\
n + \gamma_n \, \big( \beta_n + \beta_{n-1} \big) &= 0,
\end{align*}
which can be written as
\begin{align*}
{-\frac{n}{\beta_n+\beta_{n-1}}-\frac{n+1}{\beta_n+\beta_{n+1}}+\beta_n^2+t=0}.
\end{align*}

\subsection{The matrix dPI system}

We now increase further the degree of the polynomials appearing in the Pearson equations. We consider the case with $\max(h^{\mathsf L}_n (z) ,h^{\mathsf R}_n (z) )=3$, but we perform a strong simplification  
we take $h^\mathsf R=0_N$ and $ h^\mathsf L= \mu z + \nu z^3 $, with $\mu,\nu\in\C^{N\times N}$ arbitrary matrices but for $\nu$ being negative definite nonsingular matrix. Now we take $\gamma=\R$. Observe that we have non taken the more general possible polynomial of degree three, but an odd one, with well defined parity on $z$, this simplifies widely the computations.

The associated Pearson type equation for a matrix of weights of Freud type:
\begin{align}\label{eq:Pearson-dPI}
 W^{\prime} (z) = (\mu z + \nu z^3 ) W(z)
\end{align}
The struture matrix, cf.~\eqref{eq:MSP}, is a third order polynomial, that we write as follows 
\begin{align*}
 M_n (z) = M_n^0 z^3 + M_n^1 z^2 + M_n^2 z + M_n^3
\end{align*}
with 
\begin{align*}
M_n^0 &= 
\begin{bmatrix} 
\nu &0_N\\ 0_N& 0_N
\end{bmatrix} \, ,
% \\
%\end{align*}
%\begin{align*}
 & 
M_n^1 &= 
\begin{bmatrix} 
0_N & \mu \, C_n^{-1} \\ - C_{n-1} \mu & 0_N 
\end{bmatrix} \, , \\
M_n^2 &= 
\begin{bmatrix} 
\nu + [p_n^2 , \nu ] + \mu C_n^{-1} C_{n-1} & 0_N \\
0_N & -C_{n-1} \nu C_n^{-1}
\end{bmatrix} \, , 
 %\\
 &
M_n^3 &= 
%\left[
\begin{bmatrix} %{smallmatrix}
0_N & \xi_n 
%\big( \mu + [ p_n^2 , \nu ] + \nu ( C_n^{-1} C_{n-1} + C_{n+1}^{-1} C_{n}) \big)
C_n^{-1} \\
- C_{n-1} \xi_{n-1}
%\big( \mu + [ p_{n-1}^2 , \nu ] +\nu ( C_n^{-1} C_{n-1} + C_{n-1}^{-1} C_{n-2} \big) 
 & 0_N
\end{bmatrix} %{smallmatrix}
%\right]
\, , 
\end{align*}
where $\xi_n = \mu + [ p_n^2 , \nu ] + \nu ( C_n^{-1} C_{n-1} + C_{n+1}^{-1} C_{n})$, $n \in \mathbb N$.

With this at hand we find.
\begin{pro}[Matrix dPI equation]
The recursion coefficients $\gamma_n$ of the matrix orthogonal polynomials with matrix of weights satisfying the Pearson equation~\eqref{eq:Pearson-dPI} fulfill the following non-Abelian dPI equation
\begin{align*}
 \Big( \mu + \nu ( \gamma_{n+2} + \gamma_{n+1}+\gamma_{n}) + \big[ \nu, \sum_{k=1}^{n-1} \gamma_k \big] \Big) \gamma_{n+1} 
 & = -(n+1) I  , & n \in \mathbb N .
\end{align*}
\end{pro}
\begin{proof}
Compare the coefficients of $ z^{-n-1} $ in the ODE for the second kind functions we get directly (without additional computations)
the MdPI equations for the three term relation coefficients of $\big\{ P_n (z)\big\}_{n\in \N_0}$.
\end{proof}

Notice the appearance again of non local terms, that disappear if we take $\nu=-I$ and the matrix dPI reads
\begin{align*}
\gamma_{n+2} & = { n } \gamma_{n}^{-1} - \gamma_{n}-\gamma_{n-1}-\mu, & n \in \mathbb N ,
\end{align*}
which was derived in the matrix context for the first time in~\cite{CM} and the confinement of singularities for this relation was proven in~\cite{CMT,CM}, see also~\cite{GIM}.
 In 1995, Alphonse~P. Magnus~\cite{magnus:_1} for the Freud weight satisfying the Pearson equation $ W^{\prime} (z) = - \big( z^3 + 2t z \big) \, W (z) \, $ presented the following {scalar discrete Painlevé I equation}
\begin{align*}
 {\gamma_{n}\big( \gamma_{n-1} + \gamma_{n}+ \gamma_{n+1} \big) + 2 t \gamma_{n} = n}.
\end{align*}


\begin{thebibliography}{99}
 
  \bibitem{adler-van-moerbeke} M. Adler and P. van Moerbeke, \emph{Generalized orthogonal polynomials, discrete KP and
  Riemann--Hilbert problems}, Communications in Mathematical Physics \textbf{207} (1999) 589-620.
 
 \bibitem{adler-vanmoerbeke-2} M. Adler and P. van Moerbeke, \emph{Darboux transforms on band matrices, weights and
  associated polynomials}, International Mathematical Research Notices \textbf{18} (2001) 935-984.
 
 \bibitem{cum} C. Álvarez-Fernández, U. Fidalgo, and M. Mañas,
 \emph{The multicomponent 2D Toda hierarchy: generalized matrix orthogonal polynomials, multiple
  orthogonal polynomials and Riemann--Hilbert problems}, Inverse Problems \textbf{26} (2010) 055009 (17 pp.)
 
 % \bibitem{afm-2} C. Álvarez-Fernández, U. Fidalgo, and M. Ma\~{n}as, \emph{Multiple orthogonal polynomials of mixed type: Gauss-Borel factorization and the multi-component 2D Toda hierarchy}, Advances in Mathematics \textbf{227} (2011) 1451-1525.
 
 \bibitem{carlos} C. Álvarez-Fernández and M. Mañas, \emph{Orthogonal Laurent polynomials on the unit circle, extended CMV ordering and 2D Toda type integrable hierarchies}, Advances in Mathematics \textbf{240} (2013) 132-193.
 
 \bibitem{carlos2} C. Álvarez-Fernández and M. Mañas, \emph{On the Christoffel--Darboux formula for generalized matrix orthogonal polynomials of multigraded Hankel type}, Journal of Mathematical Analysis and Applications \textbf{418} (2014) 238-247.
 
 \bibitem{nuevo} C. Álvarez-Fernández, G. Ariznabarreta, J. C. García-Ardila, M. Mañas, and F. Marcellán, \emph{Christoffel transformations for matrix orthogonal polynomials in the real line and the non-Abelian 2D Toda lattice hierarchy}, 
 International Mathematics Research Notices \textbf{2017} 5 (2017) 1285–1341.
 
 \bibitem{nuevo2}  G. Ariznabarreta, J. C. García-Ardila, M. Mañas, and F. Marcellán,
 \emph{Matrix biorthogonal polynomials on the real line: Geronimus transformation}, 
submitted.
 
 \bibitem{nuevo3} G. Ariznabarreta, J. C. García-Ardila, M. Mañas, and F. Marcellán,
\emph{ Non-Abelian integrable hierarchies: matrix biorthogonal polynomials and perturbations}, 
Journal of Physics A: Mathematical and Theoretical \textbf{51} (2018) 205204 (46pp).


 
 \bibitem{baik0} J. Baik, \emph{Riemann--Hilbert problems for last passage percolation} in \emph{Recent Developments in Integrable Systems and Riemann--Hilbert Problems}, K. McLaughlin and X. Zhou, eds, Contemporary Mathematics \textbf{326} (2003) 1–21.
 
 \bibitem{baik} J. Baik, P. Deift, and K. Johansson, \emph{On the distribution of the length of the longest increasing subsequence of random permutations}, Journal of the American Mathematical Society \textbf{12} (1999) 1119–1178.

%\bibitem{baik_1}
%J. Baik, \emph{Riemann--Hilbert problems for last passage percolation in Recent Developments in Integrable Systems 
%and Riemann--Hilbert Problems}, K. McLaughlin and X. Zhou, eds, Contemporary Mathematics \textbf{326} (2003) 1-21.
% 
%\bibitem{baik_2}
%J. Baik, P. Deift, and K. Johansson, \emph{On the distribution of the length of the longest increasing subsequence of random permutations}, Journal of the American Mathematical Society b (1999) 1119-1178.

\bibitem{bere}
Ju.M. Berezanskii, \emph{Expansions in eigenfunctions of selfadjoint operators}, Transl. Math. Monographs AMS 17 (1968).

\bibitem{borrego}
J. Borrego, M. Castro, A. J. Dur\'an, \emph{Orthogonal matrix polynomials satisfying differential equations with recurrence coefficients having non-scalar limits}, 
Integral Transforms Spec. Funct. 23 (2012), no. 9, 685-700.


\bibitem{Cafasso} M. Cafasso, \emph{Matrix biorthogonal polynomials on the unit circle and non-Abelian Ablowitz-Ladik hierarchy}, Journal of Physics A: Mathematical \& Theoretical \textbf{42} (2009) 365211.

\bibitem{cantero}
M. J. Cantero, L. Moral, L. Vel\'azquez, \emph{Matrix orthogonal polynomials whose derivatives are also orthogonal},
J. Approx. Theory 146 (2007), no. 2, 174-211.

\bibitem{CM}
G. A. Cassatella-Contra and M. Mañas, \emph{Riemann--Hilbert Problems, Matrix Orthogonal Polynomials and Discrete Matrix Equations with Singularity Confinement}, Studies in Applied Mathematics \textbf{128} (2011) 252-274.
  
\bibitem{CMT} 
G. A. Cassatella-Contra, M. Mañas, and P. Tempesta, \emph{Singularity confinement for matrix discrete Painlevé equations}, Nonlinearity \textbf{27} (2014) 2321-2335,
 
\bibitem{Cassatella_3}
G. A. Cassatella-Contra and M. Mañas, \emph{Matrix biorthogonal polynomials in the unit circle: Riemann--Hilbert problem and matrix discrete Painleve II system}, \texttt{arXiv:1601.07236 [math.CA]}.

\bibitem{Clarkson_1}
P. A. Clarkson, \emph{Painlevé equations -- Nonlinear Special Functions}, 331-400 in \emph{Orthogonal Polynomials and Special Functions}, F. Marcellán, W. Van Assche (eds.), Lecture Notes in Mathematics \textbf{1883} (2011).

\bibitem{Clarkson_2}
P. A. Clarkson, A. F. Loureiro, W. Van Assche, \emph{Unique positive solution for an alternative discrete Painlevé~I equation}, Journal of Difference Equations and Applications \textbf{22} (2016) 656-675.

\bibitem{clarkson} P. A. Clarkson and K. Jordaan, \emph{The Relationship Between Semiclassical Laguerre Polynomials and the Fourth Painlevé Equation}, Constructive Approximation \textbf{29} (2014) 223-254. 

\bibitem{cresswell} C. Creswell and N. Joshi, \emph{The discrete first, second and thirty-fourth Painlevé hierarchies}, Journal of Physics A: Mathematical \& General \textbf{32} (1999) 655-669.


\bibitem{daems2} E. Daems and A. B. J. Kuijlaars, \emph{Multiple orthogonal polynomials of mixed type and non-intersecting Brownian motions}, Journal of Approximation Theory \textbf{146} (2007) 91–114.

\bibitem{daems3} E. Daems, A. B. J. Kuijlaars, and W. Veys, \emph{Asymptotics of non-intersecting Brownian motions and a $4\times 4$ Riemann–Hilbert problem}, Journal of Approximation Theory \textbf{153} (2008) 225–256.

 \bibitem{dai} D. Dai and A. B. J. Kuijlaars, \emph{Painlevé IV asymptotics for orthogonal polynomials with respect to a modified Laguerre weight}, Studies in Applied Mathematics \textbf{122} (2009) 29–83. 

\bibitem{deift1} P. A. Deift, \emph{Orthogonal Polynomials and Random Matrices: A Riemann--Hilbert Approach}, Courant Lecture Notes \textbf{3}, American Mathematical Society, Providence, RI, 2000.

\bibitem{deift2} P. A. Deift, \emph{Riemann–Hilbert methods in the theory of orthogonal polynomials, in Spectral Theory and Mathematical Physics: a Festschrift in Honor of Barry Simon’s 60th Birthday}, Proceedings of Symposia in Pure Mathematics \textbf{76}, 715–740, American Mathematical Society, Providence, RI, 2007.

\bibitem{deift5} P. A. Deift and D. Gioev, \emph{Random matrix theory: invariant ensembles and universality}, Courant Lecture Notes in Mathematics\textbf{18}, American Mathematical Society, Providence, RI, 2009.

\bibitem{deift3} P. A. Deift and X. Zhou, \emph{A steepest descent method for oscillatory Riemann–Hilbert problems. Asymptotics for the MKdV equation}, Annals of Mathematics \textbf{137} (1993) 295–368.

\bibitem{deift4} P. A. Deift and X. Zhou, \emph{Long-time asymptotics for solutions of the NLS equation with initial data in a weighted Sobolev space}, Communications in Pure Applied Mathematics \textbf{56} (2003) 1029–1077.
 
\bibitem{duran_3}
A. J. Durán, \emph{Matrix inner product having a matrix symmetric second order differential operator,} Rocky Mountain Journal of Mathematics \textbf{27} (1997) 585-600.

\bibitem{duran2004} 
A. J. Durán and F. Alberto Grünbaum, \emph{Orthogonal matrix polynomials satisfying second order differential equations}, International Mathematics Research Notices \textbf{10} (2004) 461-484.


\bibitem{duran20052} A. J. Durán and F. J. Grünbaum, \emph{Structural formulas for orthogonal matrix polynomials satisfying second order differential equations, I}, Constructive Approximation \textbf{22} (2005) 255-271.

\bibitem{duran_1} 
A. J. Durán and F. A. Grünbaum, \emph{Orthogonal matrix polynomials, scalar-type Rodrigues' formulas and Pearson equations}, Journal of
Approximation Theory \textbf{134} (2005) 267-280.


\bibitem{duran_2}
A. J. Durán and F. Alberto Grünbaum, \emph{Structural formulas for orthogonal matrix polynomials satisfying second order differential equations I}, Constructive Approximation \textbf{22} (2005) 255-271.
 
\bibitem{duran_5} 
Antonio J. Durán and Manuel D. de la Iglesia, \emph{Second order differential operators having several families of orthogonal matrix polynomials as eigenfunctions}, International Mathematics Research Notices \textbf{2008} (2008).
 
\bibitem{FIK} A. S. Fokas, A. R. Its, and A. V. Kitaev, \emph{The isomonodromy approach to matrix models in 2D
 quantum gravity}, Communications in Mathematical Physics \textbf{147} (1992) 395-430.
 

\bibitem{freud0} G. Freud, \emph{On the coefficients in the recursion formulae of orthogonal polynomials,} Proceedings of the Royal Irish Academy Section A \textbf{76} (1976) 1–6.

 
\bibitem{geronimo}
J. S. Geronimo, \emph{Scattering theory and matrix orthogonal polynomials on the real line}, Circuits Systems Signal Process. 1
(1982) 471-495.


\bibitem{GIM}
F. A. Grünbaum, M. D. de la Iglesia, and A. Martínez-Finkelshtein, \emph{ Properties of matrix orthogonal polynomials via their Riemann--Hilbert characterization}, SIGMA \textbf{7} (2011) 098.

\bibitem{hisakado} M. Hisakado, \emph{Unitary matrix models and Painlevé III}, Modern Physics Letters \textbf{A11} (1996) 3001–3010.


\bibitem{kuijlaars4} A. R. Its, A. B. J. Kuijlaars, and J. Östensson, \emph{Asymptotics for a special solution of the thirty fourth Painlevé equation}, Nonlinearity \textbf{22} (2009) 1523–1558.

\bibitem{Joshi}
N. Joshi and Y. Takei, \emph{On Stokes phenomena for the alternate discrete PI equation}, RIMS-1849 (2016).


\bibitem{Krein1} 
M.G. Krein, \emph{Infinite J-matrices and a matrix moment problem}, Dokl. Akad. Nauk SSSR 69 (2) (1949) 125-128.


\bibitem{Krein2} 
M.G. Krein, \emph{Fundamental aspects of the representation theory of hermitian operators with deficiency index $(m, m)$}, AMS
Translations, Series 2, vol. 97, Providence, Rhode Island, 1971, pp. 75-143.

\bibitem{kuijlaars2} A. B. J. Kuijlaars, A. Martínez-Finkelshtein and F. Wielonsky, \emph{Non-intersecting squared Bessel paths and multiple orthogonal polynomials for modified Bessel weights}, Communications of Mathematical Physics \textbf{286} (2009) 217–275.

\bibitem{kuijlaars3} A. B. J. Kuijlaars, \emph{Multiple orthogonal polynomial ensembles}, in \emph{Recent Trends in Orthogonal Polynomials and Approximation Theory,} Contemporary Mathematics \textbf{507}, 155–176, American Mathematical Society, Providence, RI, 2010.

\bibitem{magnus:_1} 
A. P. Magnus, \emph{Painlevé-type differential equations for the recurrence coefficients of semi-classical orthogonal polynomials}, Journal of Computational and Applied Mathematics \textbf{57} (1995) 215-237.


\bibitem{magnus} A. P. Magnus, \emph{Freud’s equations for orthogonal polynomials as discrete Painlevé equations}, in \emph{Symmetries and Integrability of Difference Equations} (Canterbury, 1996), London Mathematical Society Lecture Note Series \textbf{255}, 228–243, Cambridge University Press, Cambridge, 1999.


\bibitem{McLaughlin1} K. T.-R. McLaughlin, A. H. Vartanian, and X. Zhou, \emph{Asymptotics of Laurent polynomials of even degree orthogonal with respect to varying exponential weights}, International Mathematical Research Notices \textbf{2006} (2006).

\bibitem{McLaughlin2} K. T.-R. McLaughlin, A. H. Vartanian, and X. Zhou,\emph{ Asymptotics of Laurent polynomials of odd degree orthogonal with respect to varying exponential weights}, Constructive Approximation \textbf{27} (2008) 149–202.

 \bibitem{miranian} L. Miranian, \emph{Matrix valued orthogonal polynomials on the real line: some extensions of the classical theory}, Journal of Physics A: Mathematical and General \textbf{38} (2005) 5731-5749.

\bibitem{miranian2} L. Miranian, \emph{Matrix Valued Orthogonal Polynomials on the Unit Circle: Some Extensions of the Classical Theory}, Canadian Mathematical Bulletin \textbf{52} (2009) 95-104.

\bibitem{nikishin}
A. I. Aptekarev, E.M. Nikishin,
\emph{The scattering problem for a discrete Sturm-Liouville operator},
Math. USSR, Sb. 49, 325-355 (1984).


\bibitem{Periwal0} V. Periwal and D. Shevitz, \emph{Unitary-Matrix Models as Exactly Solvable String Theories}, Physical Review Letters \textbf{64} (1990) 1326-1329.

\bibitem{Periwal} V. Periwal and D. Shevitz, \emph{Exactly solvable unitary matrix models: multicritical potentials and correlations}, Nuclear Physics \textbf{B344} (1990) 731-746.

 \bibitem{tracy} C. A. Tracy and H. Widom, \emph{Random unitary matrices, permutations and Painlevé}, Communications in Mathematical Physics \textbf{207} (1999) 665–685.

\bibitem{VAssche} W. Van Assche, \emph{Discrete Painlevé equations for recurrence coefficients of orthogonal polynomials},
Proceedings of the International Conference on Difference Equations, Special Functions and
Orthogonal Polynomials, 687-725, World Scientific (2007).

\end{thebibliography}
\end{document}